\documentclass[12pt,reqno]{amsart}
\usepackage{xcolor}
\usepackage{amsfonts,color}
\usepackage{latexsym,amssymb}
\usepackage{amsmath}
\usepackage[utf8]{inputenc}
\usepackage{dsfont}
\usepackage{enumitem}
\usepackage{float}
\usepackage{graphicx} % Required for inserting images
\usepackage{tikz, tkz-euclide}
\usetikzlibrary{patterns,patterns,patterns.meta}

\newtheorem{theorem}{Theorem}[section]
\newtheorem{lemma}[theorem]{Lemma}

\newtheorem{remark}{Remark}[section]

\newtheorem{definition}[theorem]{Definition}

\usepackage[left=2.7cm,right=2.7cm,top=3.2cm,bottom=3.2cm]{geometry}

\usepackage[bookmarksnumbered,colorlinks, linkcolor=blue, citecolor=red, bookmarks, breaklinks]{hyperref}

\catcode`@=11 \@addtoreset{equation}{section} \catcode`@=12

\begin{document}
	\title[Hardy- Sobolev type inequalities]{Existence and Nonexistence breaking results for a weighted elliptic
problem in half-space}

\author[J. M.\ do \'O]{J. M. do \'O}
\address[Jo\~{a}o Marcos do \'O]{Dep. Mathematics, Federal University of Para\'{\i}ba
		\newline\indent
		58051-900, Jo\~ao Pessoa-PB, Brazil}
	\email{\href{mailto:jmbo@mat.ufpb.br}{jmbo@mat.ufpb.br}}
	
\author[R. F. Freire]{R. F. Freire}
\address[R. F. Freire]{Université de Pau et des Pays de l’Adour, 
Laboratoire de Mathématiques et de leurs Applications (LMAP), 
Bâtiment IPRA, Avenue de l’Université, 64013 Pau, France}
\email{\href{mailto:rdffreire@univ-pau.fr}{rdffreire@univ-pau.fr}}

\author[Giacomoni]{J. Giacomoni}
\address{Universit\'{e} de Pau et des Pays de l'Adour, Laboratoire de math\'{e}matiques et de leurs applications (LMAP), Bat. IPRA, Avenue de l’Université, 64013 Pau, France}
\email{\href{mailto: jacques.giacomoni@univ-pau.fr}{jacques.giacomoni@univ-pau.fr}}

\author[Medeiros]{E. S. Medeiros}
	\address{Universidade Federal da Para\'iba, Departamento de Matem\'atica, 58051-900 Jo\~{a}o Pessoa-PB, Brazil}
	\email{\href{mailto: everaldo@mat.ufpb.br}{everaldo@mat.ufpb.br}}
		
	%\date{\today}
	%\hyphenation{}

	\begin{abstract}  
		In this paper we study the problem $-\mathrm{div}(\rho(x_N)\nabla u)=a|u|^{p-2}u$ in $\mathds{R}^N_+$, $-\partial u/\partial x_N=b|u|^{q-2}u$ in $\mathds{R}^{N-1}$ where $a,b \in \mathds{R}$, $p,q\in (1,\infty)$ and $\rho$ is a positive weight. We establish regularity results for weak solutions and, using a variational approach combined with a new Pohozaev-type identity, we show that the introduction of the weighted operator $-\mathrm{div}(\rho(x_N)\nabla u)$ can reverse the known solvability behavior of the classical Laplacian case. Specifically, we identify regimes where the problem admits solutions despite nonexistence for the corresponding case with $-\Delta$, and vice versa, thus inverting the classical existence and nonexistence results.
	\end{abstract}

	\maketitle
	
	%%%%%%%%%%%%%%%%%%%%%%%%%%%%%%%%%%%%%%
	
	%\bigskip
	%\begin{center}
	%\begin{minipage}{8cm}
	%\footnotesize
	%\tableofcontents
	%\end{minipage}
	%\end{center}
	
	%%%%%%%%%%%%%%%%%%%%%%%%%%%%%%%%%%%%%%
	%\bigskip
	%\pagebreak
	%%%%%%%%%%%%%

%%%%%%%%%%%%%%%%%%%%%%%%%%%%%%%%%%%%%%%%%%%%%%%%%%%%%%%%%%%%%%%%%%%%%%%%%%%%%%%%%%
%
%                  INTRODUCTION
%
%%%%%%%%%%%%%%%%%%%%%%%%%%%%%%%%%%%%%%%%%%%%%%%%%%%%%%%%%%%%%%%%%%%%%%%%%%%%%%%%%%
\section{Introduction and main results}

Let $N \geq 3$ and denote by 
$\mathds{R}^N_+ := \{x=(x',x_N) \in \mathds{R}^N : x' \in \mathds{R}^{N-1}, \; x_N > 0\}$ 
the upper half-space, whose boundary is $\partial \mathds{R}^N_+ = \mathds{R}^{N-1}$.
 In this work, we are interested in investigating the regularity, existence and nonexistence of solutions to the nonlinear elliptic boundary  value problem
  \begin{equation}\label{PG}
		\left\{
		\begin{aligned}
			-\mathrm{div}(\rho(x_N) \nabla u) &=a|u|^{p-2}u &\mbox{in }&\ \mathds{R}^N_+,
			\vspace{0.2cm}\\
			-\frac{\partial u}{\partial x_N}&=b|u|^{q-2}u&\mbox{on }&\
			\mathds{R}^{N-1},
		\end{aligned}
		\right. 
		\tag{$\mathcal{P}_0$}
	\end{equation}
 where $a$ and $b$ are real constants, and the weight function $\rho$ satisfies some mild hypotheses.

For the weight function $\rho : [0,\infty) \to [0,\infty)$, we assume that $\rho \in L^1_{\mathrm{loc}}(0,\infty)$ and impose the following basic hypothesis:
\begin{enumerate}
    \item [$(\rho_0)$] There exists a constant $0\leq\gamma$ such that 
    $$
    (1+s)^\gamma\leq \rho(s),\quad s\geq0.
    $$ 
\end{enumerate}
Equations in the form $-\mathrm{div}(\rho(x_N)\nabla u)=f(u)$ naturally emerge in a variety of applied contexts, including astrophysics and population genetics (see \cite{Aronson, Chandra, Li-Ni}). In particular, the presence of a spatially dependent weight in the divergence operator models diffusion processes with non-uniform diffusion coefficients (see \cite{Lina}).

By a classical solution to \eqref{PG} we mean a function 
$
u \in C^2(\mathds{R}^N_+)\cap C^1(\overline{\mathds{R}^N_+})
$
that satisfies \eqref{PG} pointwise in $\mathds{R}^N_+$ and on its boundary $\mathds{R}^{N-1}$.
We denote the critical Sobolev exponent by $2^\ast = 2N/(N-2)$ and the critical exponent for trace embedding by $2_\ast = 2(N-1)/(N-2)$.

\begin{remark}
In \cite{YanYan-Lei}, the authors proved that under the conditions  
\[
0 \leq p \leq 2^\ast, \quad -\infty < q \leq 2_\ast, \quad p+q < 2^\ast + 2_\ast,
\]  
the problem
\begin{equation*}
\left\{
\begin{aligned}
   -\Delta u &= a |u|^{p-2}u &\quad &\text{in } \mathds{R}^N_+, \\[0.3em]
   -\frac{\partial u}{\partial x_N} &= |u|^{q-2}u &\quad &\text{on }  \mathds{R}^{N-1},
\end{aligned}
\right.
\end{equation*}
admits no positive classical solutions for any constant $a>0$.  
Moreover, they showed that if $N \geq 1$ and $p, q > 1$, then there is no positive classical solution to the problem
\begin{equation}\label{Coro3}
\left\{
\begin{aligned}
   -\Delta u + |u|^{p-2}u &= 0 &\quad &\text{in } \mathds{R}^N_+, \\[0.3em]
   -\frac{\partial u}{\partial x_N} + |u|^{q-2}u &= 0 &\quad &\text{on } \mathds{R}^{N-1}.
\end{aligned}
\right.
\end{equation}
These problems can be regarded as limit cases of \eqref{PG} with  
\[
\rho_{\lambda}(s) = (1+\lambda s)^{\gamma}, \qquad \lambda \searrow 0.
\]  
What we have discovered, however, is that this perturbation introduced by the weight function $\rho$ can actually reverse the nonexistence results and yield the existence of positive solutions; see Remark~\ref{blue} for details.
\end{remark}

One can verify that the function  
\begin{equation}\label{instanton 1}
u(x',x_N) = 
\frac{\varepsilon^{(N-2)/2}}
     {\left[\varepsilon^2 + |x^\prime|^2 + (x_N - 1)^2 \right]^{(N-2)/2}},
\qquad \varepsilon > 0,
\end{equation}
solves the problem under a nonlinear Neumann boundary condition
\begin{equation}\label{critical equation}
\left\{
\begin{aligned}
   -\Delta u &= N(N-2)\,|u|^{2^\ast-2}u 
   &\quad &\text{in } \mathds{R}^N_+, \\[0.3em]
   -\frac{\partial u}{\partial x_N} &= b\,|u|^{2_\ast-2}u 
   &\quad &\text{on } \mathds{R}^{N-1},
\end{aligned}
\right.
\end{equation}
where $(N-2)\varepsilon^{-1} = b$.  
Moreover, it was established in \cite{Escobar2, Li-Zhu} that every $C^2$ solution of \eqref{critical equation} must be of the form \eqref{instanton 1}.
When $b=0$ in \eqref{critical equation}, the equation admits the solution
\[
u(x) = C \bigl(1 + |x|^2\bigr)^{\frac{2-N}{2}}.
\]  
If $\rho \equiv 1,$  \eqref{PG} have been studied extensively; see, for example, \cite{Chipot-Fila-Shafir1,Chipot-Fila-Shafir2, YanYan-Lei} and references therein.  
In the harmonic case, we recall \cite[Corollary~1.1]{YanYan-Lei}, where the authors proved that for $N \geq 3$ and $-\infty < q < 2_\ast$, there exists no positive classical solution of  
\begin{equation}\label{Coro1}
\left\{
\begin{aligned}
	-\Delta u &= 0 &&\text{in } \mathds{R}^N_+, \\
	-\frac{\partial u}{\partial x_N} &= |u|^{q-2}u &&\text{on } \mathds{R}^{N-1}.
\end{aligned}
\right.
\end{equation}

In the critical regime, the following classical Sobolev trace inequality holds:  
\begin{equation}\label{classical trace embedding}
    S_{\partial} \left( \int_{\mathds{R}^{N-1}} |u|^{2_\ast} \, \mathrm{d}x^\prime \right)^{2/2_\ast} 
    \leq \int_{\mathds{R}^N_+} |\nabla u|^2 \, \mathrm{d}x, 
    \quad \forall\, u \in \mathcal{D}^{1,2}(\mathds{R}^N_+),
\end{equation}
where the sharp constant is given by  
\begin{equation}\label{trace best constant}
    S_{\partial} 
    = \inf_{u \in \mathcal{D}^{1,2}(\mathds{R}^N_+) \setminus \{0\}} 
    \frac{\int_{\mathds{R}^N_+} |\nabla u|^2 \, \mathrm{d}x}
    {\left( \int_{\mathds{R}^{N-1}} |u|^{2_\ast} \, \mathrm{d}x^\prime \right)^{2/2_\ast}}
    = \frac{N-2}{2} \, \omega_{N-1}^{1/(N-1)},
\end{equation}
with $\omega_{N-1}$ denoting the volume of the $(N-1)$-dimensional unit sphere.  
This inequality plays a fundamental role in the analysis of variational problems with boundary nonlinearities and provides the natural framework for studying extremal functions attaining the best constant $S_\partial$. P.-L. Lions~\cite{Lions} proved that $S_{\partial}$ is indeed achieved. Moreover, the extremal functions were independently classified by Escobar~\cite{Escobar1} and Beckner~\cite{Beckner}, and are explicitly given by  
\begin{equation}\label{instanton}
    u_\varepsilon(x',x_N) 
    = \frac{\varepsilon^{(N-2)/2}}
    {\bigl(|x'|^2 + (x_N + \varepsilon)^2\bigr)^{(N-2)/2}}.
\end{equation}
These functions $u_\varepsilon$ solve the boundary value problem  
\begin{equation}\label{Escobar solution}
	\begin{cases}
		-\Delta u = 0 & \text{in } \mathds{R}^N_+, \\[0.3em]
		-\dfrac{\partial u}{\partial x_N} = S_\partial |u|^{2_\ast-2}u 
        & \text{on } \mathds{R}^{N-1}.
	\end{cases}
\end{equation}

However, as we shall see, in the problems \eqref{critical equation} and \eqref{Escobar solution} with critical growth, perturbations can have the opposite effect: rather than generating compactness, they may lead to nonexistence of solutions. The aim of this work is to investigate how such perturbations influence the existence and regularity of solutions, in contrast with the classical Laplacian case.

Since  \eqref{PG} has the divergence form, it is very convenient to use a variational approach. Let $C^\infty_\delta(\mathds{R}^N_+)$ be the set of the functions in $C^\infty_0(\mathds{R}^N)$ restricted to $\mathds{R}^N_+$. Consider the weighted Sobolev space $\mathcal{D}^{1,2}_\rho(\mathds{R}^N_+)$ defined as the closure of 
 $C_\delta^\infty(\mathds{R}^N_+)$ with respect to the norm
	$$
		\|u\|:=\left(\int_{\mathds{R}^{N}_+}\rho(x_N)|\nabla u|^2\, \mathrm{d} x\right)^{1/2}.
	$$

In \cite{Abreu-Furtado-Medeiros 2,Do-Freire-Medeiros} the following Hardy-Sobolev type inequality is proved: Let $N\geq 2$ and $\gamma>p-1>0$. Then, for every $u \in C^{\infty}_0(\mathds{R}^N)$ it holds 
\begin{equation}\label{Hardy p}
C_{p,\gamma}^p \int_{\mathds{R}_{+}^N}\frac{|u|^p}{(1+x_N)^{p-\gamma}}\, \mathrm{d}  x  + 
			C_{p,\gamma}^{p-1} \int_{\mathds{R}^{N-1}}|u|^p\, \mathrm{d}  x'
			\leq \int_{\mathds{R}_{+}^N}(1+x_N)^\gamma |\nabla u|^p \, \mathrm{d} x,
		\end{equation}
where
$$
			C_{p,\gamma}:=\frac{\gamma-p+1}{p}.
$$
In \cite{Do-Freire-Medeiros}  was derived Sobolev-type embeddings from inequality \eqref{Hardy p} and applied to investigate existence and Liouville-type results for the indefinite quasilinear elliptic equation of the form
\begin{equation}\label{bvp}
    -\mathrm{div}(\rho(x)|\nabla u|^{p-2}\nabla u)=a(x)|u|^{q-2}u-b(x)|u|^{s-2}u, \quad \mathds{R}^N_+
\end{equation}
 with homogeneous Neumann boundary condition and $\rho_0(1+x_N)^\gamma\leq \rho(x)$ for some $\gamma, \rho_0>0$.  
 Similar results were obtained in \cite{Do-Freire-Medeiros 2} for the boundary value problem \eqref{bvp} with $a\equiv b \equiv0$ with the nonlinear boundary condition,
 \begin{equation*}
     \rho(x',0)|\nabla u|^{p-2}\nabla u\cdot\nu=h(x')|u|^{q-2}u-m(x')|u|^{s-2}u, \quad \mbox{on } \quad \mathds{R}^{N-1}.
 \end{equation*}

From inequality \eqref{Hardy p} with $p=2$, we obtain
\begin{equation}\label{Hardy}
			\frac{\gamma-1}{2} \int_{\mathds{R}^{N-1}}u^2\, \mathrm{d}  x^\prime
			\leq \int_{\mathds{R}_{+}^N}(1+x_N)^\gamma |\nabla u|^2 \, \mathrm{d} x, \quad \forall u \in 	C^\infty_\delta(\mathds{R}^N_+),
            \end{equation}
which in combination with the classical Sobolev trace inequality \eqref{classical trace embedding}, gives for some positive constant $C_0=C_0(\gamma,q)$,
\begin{equation}\label{Hardy-L}
			\left( \int_{\mathds{R}^{N-1}}|u|^q\, \mathrm{d}  x^\prime \right)^{2/q}
			\leq C_0 \int_{\mathds{R}_{+}^N}(1+x_N)^\gamma |\nabla u|^2 \, \mathrm{d} x, \quad \forall u \in 	C^\infty_\delta(\mathds{R}^N_+),
            \end{equation}
for any $ 2 \leq q \leq 2_*$. This inequality yields the Sobolev trace embedding 
\begin{equation*}
 \mathcal{D}^{1,2}_\rho(\mathds{R}^N_+) \hookrightarrow L^q(\mathds{R}^{N-1}),   \quad 2 \leq q \leq 2_*,
\end{equation*}
which, thanks to the assumption $(\rho_0)$, makes it possible to apply a variational framework in our setting. First, we establish the following result concerning the best constants for the critical regime, which is of independent interest.
\begin{theorem}\label{best constant} Let $\rho$ satisfying $(1+s)^\gamma\leq \rho(s)\leq (1+s)^\beta$ for some $0<\gamma<\beta$. Then
 \begin{equation*}\label{rho trace best constant}
   S_{2_\ast,\mathds{R}^{N-1}}:= \inf_{u \in \mathcal{D}^{1,2}_\rho(\mathds{R}^N_+)\setminus\{0\}}\frac{\int_{\mathds{R}^N_+}\rho(x_N)|\nabla u|^2\mathrm{d}x}{\left(\int_{\mathds{R}^{N-1}}|u|^{2_\ast}\mathrm{d}x^\prime\right)^{2/2_\ast}}=S_{\partial},
\end{equation*}
and $S_{2_\ast,\mathds{R}^{N-1}}$ is no achieved.
\end{theorem}

Using \eqref{Hardy p} with $p=2$ and Gagliardo-Nirenberg-Sobolev inequality and interporlation we have the continous embedding 
\begin{equation*}
   \mathcal{D}^{1,2}_\rho(\mathds{R}^N_+) \hookrightarrow L^q(\mathds{R}_+^{N}) \quad  \text{for} \quad 2 \leq q \leq 2^* \quad \text{and}  \quad \gamma \geq 2.
\end{equation*}
To exploit the cylindrical symmetry of its elements to establish compactness results we consider the space of functions in $\mathcal{D}^{1,2}_\rho(\mathds{R}^N_+)$ 
which are radial with respect to the $x'$-variable, precisely
\begin{equation}\label{R space}
    \mathcal{R}^{1,2}_\rho(\mathds{R}^N_+)
    := \Big\{\,u \in \mathcal{D}^{1,2}_\rho(\mathds{R}^N_+) : 
    u(x',x_N) = u(|x'|,x_N)\,\Big\}.
\end{equation}

By employing properties of Schwarz symmetrization, we prove the following:

\begin{theorem}\label{Compact embedding}
Assume $(\rho_0)$ with $\gamma>2$. Then, the embedding 
\begin{equation*}
		\mathcal{R}^{1,2}_\rho(\mathds{R}^N_+)\hookrightarrow\left\{
		\begin{aligned}
			L^p(\mathds{R}^N_+),~p\in (2,2^\ast)\\
			L^q(\mathds{R}^{N-1}),~q\in (2,2_\ast),
		\end{aligned}
		\right.
	\end{equation*}
is compact.
\end{theorem}

Using \eqref{Hardy p} with $p=2$, together with the Gagliardo--Nirenberg--Sobolev inequality and interpolation, we obtain the continuous embedding
\[
\mathcal{D}^{1,2}_\rho(\mathds{R}^N_+) \hookrightarrow L^q(\mathds{R}^N_+),
\qquad 2 \leq q \leq 2^*, \; \gamma \geq 2.
\]

As an application of the compactness result above, we establish the attainability of certain best constants:

\begin{theorem}\label{Attainability}
Assume $(\rho_0)$ with $\gamma > 2$, $q \in (2, 2_\ast)$, and $p \in (2, 2^\ast)$. Then, the following best constants are attained
\begin{equation}\label{Constant1}
S_{q,\mathds{R}^{N-1}} := \inf_{u \in \mathcal{D}^{1,2}_\rho(\mathds{R}^N_+) \setminus \{0\}} \frac{\int_{\mathds{R}^N_+} \rho(x_N) |\nabla u|^2  \mathrm{d}x}{\left( \int_{\mathds{R}^{N-1}} |u|^{q}  \mathrm{d}x^\prime \right)^{2/q}},
\end{equation}
\begin{equation}\label{Constant2}
S_{p,\mathds{R}^N_+} := \inf_{u \in \mathcal{D}^{1,2}_\rho(\mathds{R}^N_+) \setminus \{0\}} \frac{\int_{\mathds{R}^N_+} \rho(x_N) |\nabla u|^2  \mathrm{d}x}{\left( \int_{\mathds{R}^N_+} |u|^p  \mathrm{d}x \right)^{2/p}}.
\end{equation}
Furthermore, for $\gamma > 1$ and $q \in [2, 2_\ast]$, we have $S_{q,\mathds{R}^{N-1}} > 0$, and this infimum is not attained when $q = 2$. Similarly, for $\gamma \geq 2$ and $p \in [2, 2^\ast]$, we have $S_{p,\mathds{R}^N_+} > 0$, and the infimum is not attained when $p = 2$.
\end{theorem}

\begin{remark} 
It is worth emphasizing that, under suitable additional assumptions, the conclusions of Theorem~\ref{Attainability} are in fact sharp. 
In the framework of Theorem~\ref{best constant}, one also finds that $S_{2_\ast,\mathds{R}^{N-1}}$ is not attained. 
Furthermore, under the extra assumption $(\rho_1)$ in Theorem~\ref{Attainability}, Theorem~\ref{Nonexistence1} shows that $S_{2^\ast,\mathds{R}^N_+}$ is not attained either.    
\end{remark}

\begin{definition} By a weak solution of \eqref{PG} we mean a function  
$ u \in \mathcal{D}^{1,2}_\rho(\mathds{R}^N_+)$
such that  
\begin{equation}\label{weak solution}
   \int_{\mathds{R}^{N}_+} \rho(x_N)\,\nabla u \cdot \nabla \varphi \, \mathrm{d}x
   = a \int_{\mathds{R}^{N}_+} |u|^{p-2}u \varphi \, \mathrm{d}x  
   + b \int_{\mathds{R}^{N-1}} |u|^{q-2}u \varphi \, \mathrm{d}x^\prime,
\end{equation}
for every testing function $\varphi \in C_\delta^\infty(\mathds{R}^N_+)$.  
For simplicity, and without loss of generality, we assume throughout this work that $\rho(0)=1$.

\end{definition}

It is standard to prove that weak solutions of problem~\eqref{PG} 
coincide with the critical points of the associated energy functional.
$I:\mathcal{D}^{1,2}_\rho(\mathds{R}^N_+) \to \mathds{R}$ defined by
\begin{equation*}
I(u)=\frac{1}{2}\int_{\mathds{R}^N_+}\rho(x_N)|\nabla u|^2\mathrm{d}x-\frac{b}{q}\int_{\mathds{R}^{N-1}}|u|^q\mathrm{d}x^\prime-\frac{a}{p}\int_{\mathds{R}^N_+}|u|^p\mathrm{d}x.
\end{equation*}

\subsection{Regularity results} We establish regularity results for \eqref{PG}, with explicit dependence on all relevant parameters. Our interest is describe when a weak solution can be promoted to a classical one, as well as to identify the minimal regularity required to recover the boundary condition and derive nonexistence results via Pohozaev-type identities. For that we will consider the following spaces 
\begin{equation}
    \begin{aligned}
        E_p(\mathds{R}^N_+) :=\, & \overline{C^\infty_\delta(\mathds{R}^N_+)}^{\|u\|_{E_p(\mathds{R}^N_+)}}\\ 
        E_q(\mathds{R}^{N-1}) :=\, & \overline{C^\infty_\delta(\mathds{R}^N_+)}^{\|u\|_{E_q(\mathds{R}^{N-1})}}
    \end{aligned}
\end{equation}
for $p,q\in(1,\infty),$ where 
\begin{equation*}   
\|u\|_{E_p(\mathds{R}^N_+)}=\|u\|+\|u\|_{p,\mathds{R}^N_+}\quad\text{and}\quad  \|u\|_{E_q(\mathds{R}^{N-1})}=\|u\|+\|u\|_{q,\mathds{R}^{N-1}}.
\end{equation*}
 By Lemma \ref{embedding}, $E_p(\mathds{R}^N_+)=\mathcal{D}^{1,2}_\rho(\mathds{R}^N_+)$ when $\gamma\geq2$ and $p\in[2,2^\ast]$ and similarly, $E_q(\mathds{R}^{N-1})=\mathcal{D}^{1,2}_\rho(\mathds{R}^N_+)$ when $\gamma>1$ and $q\in[2,2_\ast]$.

\bigskip

Concerning Holder regularity, we prove the following result:

\begin{theorem}\label{Holder regularity result}
Assume $(\rho_0)$ and that $\rho\in C^{1,\alpha}_{\mathrm{loc}}[0,\infty)$. In each of the following cases:
\begin{itemize}
    \item[$(i)$] $\gamma \geq 0$, $a > 0$, $b \leq 0$, $p = 2^\ast$, and $q \in (1,\infty)$;
    \item[$(ii)$] $\gamma \geq 2$, $a > 0$, $b \leq 0$, $p \in [2, 2^\ast)$, and $q \in (1,\infty)$;
\end{itemize}
if $u$ is a nonnegative weak solution in $E_q(\mathds{R}^{N-1})$, then \( u \in C^{2,\alpha}_{\mathrm{loc}}(\overline{\mathds{R}^N_+}) \) and \( u > 0 \) in \( \overline{\mathds{R}^N_+} \). Similarly, in each of the following cases:
\begin{itemize}
    \item[$(iii)$] $\gamma \geq 0$, $a \leq 0$, $b > 0$, $p \in (1,\infty)$, and $q = 2_\ast$;
    \item[$(iv)$] $\gamma > 1$, $a \leq 0$, $b > 0$, $p \in (1,\infty)$, and $q \in [2, 2_\ast)$;
\end{itemize}
if $u$ is a nonnegative weak solution in $E_p(\mathds{R}^N_+)$, then \( u \in C^{2,\alpha}_{\mathrm{loc}}(\overline{\mathds{R}^N_+}) \) and $u>0$ in $ \overline{\mathds{R}^N_+} $. Furthermore, if $\gamma\geq 2$, $a,b>0$ and $u$ is a nonnegative weak solution in $\mathcal{D}^{1,2}_\rho(\mathds{R}^N_+)$ with $p\in[2,2^\ast)$ and $q\in[2,2_\ast)$, then $u \in C^{2,\alpha}_{\mathrm{loc}}(\mathds{R}^N_+)$ and $u>0$ in $\mathds{R}^N_+$.
\end{theorem}

Without requiring Hölder continuity of $\rho^\prime$, we can still obtain $H^2_{\mathrm{loc}}(\overline{\mathds{R}^N_+)}$-regularity. This level of regularity is sufficient to recover the boundary condition for weak solutions and to establish the Pohozaev identity. 
The following result holds:

\begin{theorem}\label{H2_loc regularity theorem} Assume $(\rho_0)$ with $\rho\in C^1[0,\infty)$. Then, in each of following assertions:
\begin{itemize}
    \item[$(i)$]  $u$ is a nonnegative weak solution in $E_q(\mathds{R}^{N-1})$, with $\gamma\geq 2$, $a > 0$, $b \leq 0$, $p \in [2,2^\ast]$, and $q \in (1,\infty)$;
    \item[$(ii)$]  $u$ is a nonnegative weak solution in $E_p(\mathds{R}^N_+)$, with $\gamma>1$, $a \leq 0$, $b > 0$, $p \in (1,\infty)$, and $q \in (2, 2_\ast]$;
\end{itemize}
we have $u\in H^2_{\mathrm{loc}}(\overline{\mathds{R}^N_+})$. Furthermore, the same conclusion holds in case $(i)$ for $\gamma \geq 0$ and $p = 2^*$.
\end{theorem}

\subsection{Existence results}
 
First, we investigate the results of existence for \eqref{PG} with $a\leq 0$ and $b>0$. In this case, the result obtained is as follows:

\begin{theorem}\label{Problem1-Existence} Assume $(\rho_0)$ with $\gamma>1$, $a\leq0$, $b>0$ and suppose that $1<p$ and $\max\{2,p\}<q<2_\ast$. Then problem \eqref{PG} has a nontrivial weak solution in $E_p(\mathds{R}^N_+)$, which is positive. Furthermore, if $\rho \in C^{1,\alpha}_{\mathrm{loc}}[0,\infty)$, $u$ belongs $C^{2,\alpha}_{\mathrm{loc}}(\overline{\mathds{R}^N_+})$ and is positive in $\overline{\mathds{R}^N_+}$. 
\end{theorem} 

\begin{remark} \label{blue}
As previously observed, in the case $\rho \equiv 1$, $a=0$ and $b=1$, there is no solution for \eqref{PG} when $-\infty<q<2_\ast$, while there are solutions when $q=2_\ast$. By Theorem~\ref{Nonexistence1}, we conclude that the problem~\eqref{PG}, with $a\leq 0$ and $b>0$ admits no weak solution when $\gamma>1$, $p\in (1,2^\ast]$ and $q=2_\ast$. Therefore, the introduction of the weight function $\rho$ is precisely what makes the existence of solutions possible in the subcritical case and leads to nonexistence when $q=2_\ast$.
\end{remark}

\begin{remark}
It is worth to observe that considering the changes of variable 
 $
 u=v/{(1+x_N)} 
 $
(see \cite{Abreu-Furtado-Medeiros}) and using a straightforward computation we can see that problem \eqref{PG} with $a=0$ and $\rho(x_N)=(1+x_N)^2$,  is equivalent to the following problem with nonlinear Robin boundary condition
\begin{equation}\label{R2}
		\left\{
		\begin{aligned}
			\Delta v&=0 &\mbox{in }&\ \mathds{R}^N_+,
			\vspace{0.2cm}\\
			\frac{\partial v}{\partial \nu}+v&=b|v|^{q-2}v&\mbox{on }&\
			\mathds{R}^{N-1},
		\end{aligned}
		\right. 
	\end{equation}
which was studied in \cite{Abreu-JM-Medeiros}. In particular, from the results in \cite{Abreu-JM-Medeiros}, we infer that the solutions in this case decays as
$$
u(x^\prime,x_N)\leq \frac{c_1 e^{-c_2|x^\prime|}}{(1+x_N)(1+x_N^2)^{(N-2)/2}}, 
$$
for all $(x^\prime,x_N)\in \overline{\mathds{R}^N_+}$ and some positive constants $c_3,c_4$.
\end{remark}

In the following existence results, we consider weak solutions with cylindrical symmetry. In this case we have:

\begin{theorem}\label{Problem2-Existence} Let $a,b>0$ and suppose that $(\rho_0)$ holds with $\gamma>2$. Then problem \eqref{PG} has a nontrivial weak solution in $\mathcal{R}^{1,2}_\rho(\mathds{R}^N_+)$, whenever $p\in (2,2^\ast)$ and $q \in (2,2_\ast)$. Furthermore, if $\rho \in C^{1,\alpha}_{\mathrm{loc}}(0,\infty)$, then the weak solution  $u$ belongs to $C^{2,\alpha}_{\mathrm{loc}}(\mathds{R}^N_+)$ and is positive in $\mathds{R}^N_+$. The same can be obtained when
\begin{itemize}
    \item [$(i)$]  $p=2$, $q\in(2,2_\ast)$, $b>0$, $0<a<(\gamma-1)^2/4$;
    \item[$(ii)$] $ p\in(2,2^\ast)$, $q=2$, $a>0$, $0<b<(\gamma-1)/2$.
\end{itemize}
\end{theorem}

In the cases where $q=2$ and $p\in(2,2^\ast)$ or $p=2$ and $q\in(2,2_\ast)$, respectively, we work with the norms: $(\|\cdot\|^2-b\|\cdot\|_{2,\mathds{R}^{N-1}}^2)^{1/2}$ and $(\|\cdot\|^2-a\|\cdot\|_{2,\mathds{R}^N_+}^2)^{1/2}$, which are equivalent to $\|\cdot\|$ when $\gamma\geq 2$, $0<b<(\gamma-1)/2$, and $0<a<(\gamma-1)^2/4$, due to $\eqref{Hardy p}$. Finally, in the case $a>0$ and $b\leq 0$, our existence result is the following:

\begin{theorem}\label{Existence 3}
Assume condition $(\rho_0)$ with $\gamma > 2$, $a > 0$, $b \leq 0$, and $\max\{2,q\} < p < 2^\ast$ with $1<q < 2_\ast$. Then there exists a weak solution $u \in \mathcal{R}^{1,2}_\rho(\mathds{R}^N_+)\cap E_q(\mathds{R}^{N-1})$ to the problem \eqref{PG}. Moreover, if $\rho \in C^{1,\alpha}_{\mathrm{loc}}[0,\infty)$, then $C^{2,\alpha}_{\mathrm{loc}}(\overline{\mathds{R}^N_+})$ for some $\alpha \in (0,1)$, and $u > 0$ in $\overline{\mathds{R}^N_+}$.
\end{theorem}

\begin{remark}\label{N=2} For the case $N=2$, it was proved in \cite[Proposition 6.1]{Chipot-Fila-Shafir1} that if 
\begin{equation}\label{Problem N=2}
			\left\{
			\begin{aligned}
				\Delta u&\leq 0&\quad&\text{ in }\quad\mathds{R}^2_+,\\
				\frac{\partial u}{\partial x_N}&\leq 0&\quad&\text{ on }\quad\mathds{R},\\
                u&\geq0&\quad& \text{ in }\quad\mathds{R}^2_+,
			\end{aligned}
			\right.
		\end{equation}
 then $u$ is constant. To treat \eqref{PG} with $N=2$, we can consider the Sobolev inequality
\begin{equation*}
    \left(\int_{\mathds{R}^2_+}\frac{|u|^p}{(1+x_N)^{2-\gamma}}\mathrm{d}x\right)^{2/p}\leq C\int_{\mathds{R}^2_+}(1+x_N)^\gamma|\nabla u|^2\mathrm{d}x,
\end{equation*}
for $\gamma>1$ and $p \in [2,\infty)$(see \cite[Theorem 2.12]{Do-Freire-Medeiros}). By combining this inequality with the approach developed in \cite[Theorem 1.5]{AFM}, we can derive the following trace inequality:
\begin{equation*}
    \left(\int_{\mathds{R}}|u|^q\mathrm{d}x^\prime\right)^{2/q}\leq C\int_{\mathds{R}^2_+}(1+x_N)^\gamma|\nabla u|^2\mathrm{d}x, 
\end{equation*}
for $\gamma>1$ and $q\in [2,\infty)$. With these inequalities, we can establish a version of the embeddings proved in Lemma~\ref{embedding} for the case $N=2$, and apply the same strategy used in Theorems~\ref{Problem1-Existence}, \ref{Problem2-Existence}, and~\ref{Existence 3} to obtain the existence of weak solutions in the corresponding situation for dimension $2$. In particular, the cases $a= 0$, $b>0$ with $\gamma>1$ and $a>0$, $b=0$ with $\gamma>2$ in \eqref{PG} will stand in contrast to the result for \eqref{Problem N=2}.
\end{remark}

\subsection{Pohozaev identity and nonexistence results} As is well known, Pohožaev-type identities provide necessary conditions for the existence of solutions to elliptic problems, making them a powerful tool for establishing nonexistence results. Such identities have been explored in various contexts: see \cite{Berestycki-Lions, Esteban-Lions, Pohozaev} for Laplacian problems in bounded domains, the whole space, and domains with a "half-space" structure. For problems involving the \( p \)-Laplacian, see \cite{CAOM, Guedda-Veron, Ilyasov-Takac, Liu-Liu}; see also \cite{Pucci-Serrin} for a more general quasilinear framework.

To prove our Pohozaev-type identity and the corresponding nonexistence results, we assume the following additional condition on the weight function $\rho$:
\begin{enumerate}
\item [$(\rho_1)$]$\rho \in C^1[0,\infty)$ and there exists a constant $c_1>0$ such that 
$$
0< \rho^\prime(s)s\leq c_1 \rho(s), \quad s>0.
$$
\end{enumerate}
Based on some ideas presented in \cite{CAOM, Ilyasov-Takac}, we prove the following:

    \begin{theorem}[Pohozaev identity]\label{Pohozaev} Assume $(\rho_1)$ and let $f,g:\mathds{R}\longrightarrow\mathds{R}$ be continuous functions. If $u\in  H^2_{\mathrm{loc}}(\overline{\mathds{R}^N_+})$ is a weak solution of the problem
		\begin{align*}
			\left\{
			\begin{aligned}
				-\mathrm{div}(\rho(x_N) \nabla u)&=f(u)&\quad&\text{ in }\quad\mathds{R}_{+}^{N},\\
				\frac{\partial u}{\partial \nu}&=g(u)&\quad&\text{ on }\quad\mathds{R}^{N-1},
			\end{aligned}
			\right.
		\end{align*}
		satisfying $\rho(x_N)|\nabla u|^2\in L^1(\mathds{R}^N_+)$, $F(u)\in L^1(\mathds{R}^N_+)$, and $G(u)\in L^1(\mathds{R}^{N-1})$  where 
		\[
		F(t)=\int_0^tf(s)ds, \ \ \mbox{ and }\ \ G(t)=\int_0^tg(s)ds,
		\]
		then, $u$ satisfies the following Pohozaev type identity
		\[
		\frac{N-2}{2}\int_{\mathds{R}_{+}^{N}}\rho(x_N)|\nabla u|^2\mathrm{d}x+\frac{1}{2}\int_{\mathds{R}_{+}^{N}}\rho'(x_N)x_N|\nabla u|^2\mathrm{d}x=N\int_{\mathds{R}^N_+}F(u)\mathrm{d}x+(N-1)\int_{\mathds{R}^{N-1}}G(u)\mathrm{d}x^\prime.
		\]
	\end{theorem}

As an application of our Pohozaev identity, we obtain the following nonexistence result:

\begin{theorem}\label{Nonexistence1}
Assume $(\rho_0)$ and $(\rho_1)$. Then $u\equiv 0$ in the following cases:
\begin{itemize}
\item[$(i)$] $u$ is a nonnegative weak solution in $E_q(\mathds{R}^{N-1})$, with $\gamma > 0$, $a > 0$, $b \leq 0$, $p = 2^\ast$, and $q \in (1, 2_\ast]$;
\item[$(ii)$] $u$ is a nonnegative weak solution in $E_p(\mathds{R}^N_+)$, with $\gamma > 1$, $a \leq 0$, $b > 0$, $p \in (1, 2^\ast]$, and $q = 2_\ast$.
\end{itemize}
\end{theorem}

\begin{theorem}\label{Nonexistence2}Assume $(\rho_0)$ with $\gamma>0$ and $(\rho_1)$. Let $u\in \mathcal{D}^{1,2}_\rho(\mathds{R}^N_+)\cap H^2_{\mathrm{loc}}(\overline{\mathds{R}^N_+})$ be a nonnegative weak solution of \eqref{PG}. Then $u\equiv 0$ in the following cases:
\begin{itemize}
\item[$(i)$] $u\in E_q(\mathds{R}^{N-1})$, $a > 0$, $b \leq 0$, and $p = 2^\ast$ and $q\in(1,2_\ast]$;
\item[$(ii)$] $u\in E_p(\mathds{R}^N_+)$, $a \leq 0$, $b > 0$, $p\in(1,2^\ast]$ and $q = 2_\ast$;
\item [$(iii)$] $a>0$, $b>0$, $p= 2^\ast$ and $q= 2_\ast$.
\end{itemize}
\end{theorem}

\begin{remark}
In item $(iii)$, we encounter the case where the perturbation introduced by the weight function $\rho$ leads to nonexistence for problem \eqref{critical equation}, despite the existence and classification results previously established for it.
\end{remark}

Inspired by the present work, we propose several potential directions for future research:

\begin{itemize}
\item Our results establish existence and nonexistence of solutions based on the newly developed embeddings derived from inequality \eqref{Hardy p}. However, certain interesting scenarios remain to be explored. For example, a promising direction would be to analyze the problem \eqref{PG} under conditions $a \leq 0$, $b > 0$, $q \in (2, 2_\ast)$, and $\gamma \in (0,1]$. Part of the difficulties of the problem arises from the fact that for $\rho(x_N)=(1+x_N)^\gamma$ with $\gamma<1$, $S_{2,\mathds{R}^{N-1}}=0$ (see \cite[Theorem 2]{Do-Freire-Medeiros 2}).

\item Another open case arises when $a, b > 0$ and either $p$ or $q$ reaches the critical exponent; that is, when $p = 2^\ast$ and $q \in (1, 2_\ast)$, or $p \in (1, 2^\ast)$ and $q = 2_\ast$.
\end{itemize}

The paper is organized as follows. In Section 2, we present the functional setting, the key embeddings and the attainability results. Section 3 is dedicated to proving the regularity of nonnegative weak solutions, culminating in Theorems \ref{Holder regularity result} and \ref{H2_loc regularity theorem}. In Section 4, we present the main existence results using a variational approach. Finally, in Section 5, we prove the Pohozaev-type identity and use it to establish the nonexistence results. 
\bigskip 

%%%%%%%%%%%%%%%%%%%%%%%%%%%%%%%%%%%%%%%%%%%%%%%%%%%%%%%%%%%%%%%%%%%%%%%%%%%%%%%%%%r
%
%                  PRLIMINARY RESULTS
%
%%%%%%%%%%%%%%%%%%%%%%%%%%%%%%%%%%%%%%%%%%%%%%%%%%%%%%%%%%%%%%%%%%%%%%%%%%%%%%%%%%

\section{Preliminary results}

The main purpose of this section is to establish the embedding results that are necessary to apply our approach to the problem \eqref{PG} and prove the Theorems \ref{best constant}, \ref{Compact embedding} and \ref{Attainability}.

\subsection{Proof of Theorem \ref{best constant}:}

\begin{lemma}\label{embedding}
Assume $(\rho_0)$ with $\gamma>1$. Then, the following embedding holds
$$
\mathcal{D}_\rho^{1,2}(\mathds{R}^N_+)\hookrightarrow L^q(\mathds{R}^{N-1})\quad \mbox{for all}\quad q\in[2,2_\ast]. 
$$

Furthermore, if $\gamma\geq 2$, then the following embedding holds
$$
\mathcal{D}_\rho^{1,2}(\mathds{R}^N_+)\hookrightarrow H^1(\mathds{R}^N_+)\hookrightarrow L^p(\mathds{R}^{N}_+)\quad \mbox{for all}\quad p \in [2,2^\ast].
$$
\end{lemma}

\begin{proof} By \eqref{Hardy}, we have
\begin{equation*}
\frac{\gamma - 1}{2} \int_{\mathds{R}^{N-1}} u^2 \mathrm{d}x
\leq \int_{\mathds{R}{+}^N} (1 + x_N)^\gamma |\nabla u|^2 \mathrm{d}x.
\end{equation*}
On the other hand, from the classical trace inequality \eqref{classical trace embedding}, we obtain
\begin{equation}\label{ineq}
S_{\partial} \left( \int_{\mathds{R}^{N-1}} |u|^{2_\ast}  \mathrm{d}x \right)^{2/2_\ast}
\leq \int_{\mathds{R}^N_+} |\nabla u|^2  \mathrm{d}x
\leq \int_{\mathds{R}^N_+} (1 + x_N)^\gamma |\nabla u|^2  \mathrm{d}x^\prime.
\end{equation}
By applying an interpolation argument in conjunction with these two inequalities and the hypothesis $(\rho_0)$, we conclude the trace embedding. If $\gamma\geq 2$ it follows by \eqref{Hardy p} that
\begin{equation*}
    \int_{\mathds{R}^N_+}[|\nabla u|^2+u^2]\mathrm{d}x\leq \int_{\mathds{R}^N_+}\left[(1+x_N)^\gamma|\nabla u|^2+\frac{u^2}{(1+x_N)^{2-\gamma}}\right]\mathrm{d}x\leq C_\gamma \int_{\mathds{R}^N_+}(1+x_N)^\gamma|\nabla u|^2\mathrm{d}x,
\end{equation*}
then, using $(\rho_0)$ again we conclude the second embedding.
\end{proof}
%%%%%%%%%%%%%%%%%%%%%%%%%%%%%%%%%%%%%%%%%%%%%%%%%%%%%%%%%%%%%

We will consider the balls centered on $y\in\mathds{R}^{N-1}$  with radius $r,$
	\[
	\quad\Gamma_r(y)=\left\{ x^\prime\in\mathds{R}^{N-1}\text{ : }|x^\prime-y|<r \right\}\quad\text{and}\quad B_{r}^+(y)=\left\{ x\in \mathds{R}_{+}^{N}\text{ : } |x-y|<r  \right\}.
	\]
When $y=0$, we use the notation $B_{r}^+=B_{r}^+(0)$, $\Gamma_r=\Gamma_r(0)$

\begin{proof}[Proof of Theorem \ref{best constant}:]
We first observe that by \eqref{ineq} we have $S_{\partial}\leq S_{2_\ast,\mathds{R}^{N-1}}$. In order to show that $S_{2_\ast,\mathds{R}^{N-1}}\leq S_{\partial}$, let $r>0$ and $\varphi_r \in C^\infty_0(\mathds{R}^N)$ such that
\begin{equation*}
		\varphi_r=\left\{
		\begin{aligned}
			1, \quad&\mbox{in }&\ B_{r/2}^+(0),
			\vspace{0.2cm}\\
			0, \quad &\mbox{in }&\mathds{R}^{N}_+\backslash B_r(0).
		\end{aligned}
		\right. 
	\end{equation*}
For any  $\varepsilon>0$ set $\psi_\varepsilon:= \varphi_r u_\varepsilon$, where $u_\varepsilon$ is given by \eqref{instanton}, that is,
\[
\psi_\varepsilon(x)=\frac{\varepsilon^{(N-2)/2}\varphi_r}{[|x^\prime|^2+(x_N+\varepsilon)^2]^{(N-2)/2}}.
\]
Initially we see that
\[
\nabla \psi_\varepsilon = \frac{\varepsilon^{(N-2)/2}\nabla \varphi_r}{[|x^\prime|^2+(x_N+\varepsilon)^2]^{(N-2)/2}}-\frac{\varepsilon^{(N-2)/2}(N-2)\varphi_r(x^\prime,x_N+\varepsilon)}{[|x^\prime|^2+(x_N+\varepsilon)^2]^{N/2}},
\]
thus
\begin{align*}
    \int_{\mathds{R}^N_+}|\nabla \psi_\varepsilon|^2\mathrm{d}x=& \int_{\mathds{R}^N_+}\frac{\varepsilon^{N-2}|\nabla \varphi_r|^2}{[|x^\prime|^2+(x_N+\varepsilon)^2]^{N-2}}\mathrm{d}x-2(N-2)\int_{\mathds{R}^N_+}\frac{\varepsilon^{N-2}\varphi_r \nabla \varphi_r(x^\prime,x_N+\varepsilon)}{[|x^\prime|^2+(x_N+\varepsilon)^2]^{N-1}}\mathrm{d}x\\
    +&(N-2)^2\int_{\mathds{R}^N_+} \frac{\varepsilon^{N-2}\varphi_r^2}{[|x^\prime|^2+(x_N+\varepsilon)^2]^{N-1}}\mathrm{d}x\\
    &= I_1-2(N-2)I_2+I_3.
\end{align*}
For $I_1$, we have
\begin{align*}
    I_1 = \varepsilon^{N-2}\int_{B_r^+\backslash B_{r/2}}\frac{|\nabla \varphi_r|^2}{[|x^\prime|^2+(x_N+\varepsilon)^2]^{N-2}}\mathrm{d}x\leq C \varepsilon^{N-2}\int_{B_r^+\backslash B_{r/2}}\frac{\mathrm{d}x}{[|x^\prime|^2+(x_N+\varepsilon)^2]^{N-2}}, 
\end{align*}
thus $I_1=O(\varepsilon^{N-2})$. Similarly, $I_2=O(\varepsilon^{N-2})$.

\begin{align*}
    I_3=&\varepsilon^{N-2}\int_{\mathds{R}^N_+}\frac{\mathrm{d}x}{[|x^\prime|^2+(x_N+\varepsilon)^2]^{N-1}}-\varepsilon^{N-2}\int_{\mathds{R}^N_+\setminus B_{r/2}}\frac{\mathrm{d}x}{[|x^\prime|^2+(x_N+\varepsilon)^2]^{N-1}}\\
    &+\varepsilon^{N-2}\int_{\mathds{R}^{N}_+\backslash B^+_{r/2}}\frac{\varphi_r^2}{[|x^\prime|^2+(x_N+\varepsilon)^2]^{N-1}}\mathrm{d}x\\
    =& \|\nabla u_\varepsilon\|_2^2+O(\varepsilon^{N-2}).
\end{align*}
Then, $\|\nabla \psi_\varepsilon\|_2^2= \|\nabla u_\varepsilon\|_2^2+O(\varepsilon^{N-2})$. On the other hand, 
\begin{align*}
    \int_{\mathds{R}^{N-1}}|\psi_\varepsilon|^{2_\ast}\mathrm{d}x^\prime&=\varepsilon^{N-1}\int_{\mathds{R}^{N-1}}\frac{|\varphi_r|^{2_\ast}}{[|x^\prime|^2+\varepsilon^2]^{N-1}}\mathrm{d}x^\prime\\
    &= \varepsilon^{N-1}\int_{\mathds{R}^{N-1}}\frac{\mathrm{d}x^\prime}{[|x^\prime|^2+\varepsilon^2]^{N-1}}-\varepsilon^{N-1}\int_{\mathds{R}^{N-1}\setminus \Gamma_{r/2}}\frac{\mathrm{d}x^\prime}{[|x^\prime|^2+\varepsilon^2]^{N-1}} \\
    &+\varepsilon^{N-1}\int_{\mathds{R}^{N-1}\backslash \Gamma_{r/2}}\frac{|\varphi_r|^{2_\ast}}{[|x^\prime|^2+\varepsilon^2]^{N-1}}\mathrm{d}x^\prime\\
    &= \|u_\varepsilon\|_{2_\ast, \mathds{R}^{N-1}}^{2_\ast}+O(\varepsilon^{N-1}).
\end{align*}
Thus,
\begin{equation*}
    \frac{\int_{\mathds{R}^N_+}|\nabla \psi_\varepsilon|^2\mathrm{d}x}{\left(\int_{\mathds{R}^{N-1}}|\psi_\varepsilon|^{2_\ast}\mathrm{d}x^\prime\right)^{2/2_\ast}}= \frac{\|\nabla u_\varepsilon\|_2^2+O(\varepsilon^{N-2})}{\left(\|u_\varepsilon\|_{2_\ast,\mathds{R}^{N-1}}^{2\ast}+O(\varepsilon^{N-1})\right)^{2/2_\ast}}=S_\partial+O(\varepsilon^{N-2}).
\end{equation*}
Therefore,
\begin{equation*}
    S_{2_\ast,\mathds{R}^{N-1}}\leq \frac{\int_{\mathds{R}^{N}_+}\rho(x_N)|\nabla \psi_\varepsilon|^2\, \mathrm{d} x}{\left(\int_{\mathds{R}^{N-1}}|\psi_\varepsilon|^{2_\ast}\mathrm{d}x^\prime\right)^{2/2_\ast}}\leq \frac{(1+r)^\beta\int_{\mathds{R}^{N}_+}|\nabla \psi_\varepsilon|^2\, \mathrm{d} x}{\left(\int_{\mathds{R}^{N-1}}|\psi_\varepsilon|^{2_\ast}\mathrm{d}x^\prime\right)^{2/2_\ast}}= (1+r)^\beta S_{\partial}+O(\varepsilon^{N-2})
\end{equation*}
and passing to the limit as $r, \varepsilon \longrightarrow 0^+$, we obtain $S_{2_\ast,\mathds{R}^{N-1}}\leq S_{\partial}$. 

Now suppose that $S_{2_\ast,\mathds{R}^{N-1}}$ is attained by a function $u$. Then, we have
\[
 S_\partial= S_{2_\ast,\mathds{R}^{N-1}}\geq   \frac{\int_{\mathds{R}^N_+}(1+x_N)^\gamma|\nabla u|^2\mathrm{d}x}{\left(\int_{\mathds{R}^{N-1}}|u|^{2_\ast}\mathrm{d}x^\prime\right)^{2/2_\ast}}>  \frac{\int_{\mathds{R}^N_+}|\nabla u|^2\mathrm{d}x}{\left(\int_{\mathds{R}^{N-1}}|u|^{2_\ast}\mathrm{d}x^\prime\right)^{2/2_\ast}},
\]
which is a contradiction.
\end{proof}

\bigskip
\subsection{Proof of Theorem \ref{Compact embedding} and \ref{Attainability}} 

Given a function $u$, let us denote by $u^\ast(|x^\prime|, x_N)$, the nonincreasing rearrangement of the function $u(., x_N)$ in $\mathds{R}^{N-1}$ as given by its Schwarz symmetrization. The following properties hold for $s\in[1,\infty]$(see \cite{Badiale-Tarantello, Kawohl, PLLions}):

\begin{equation}\label{Property i}
 \int_{\mathds{R}^{N-1}}|u^\ast|^s\mathrm{d}x^\prime=\int_{\mathds{R}^{N-1}}|u|^s\mathrm{d}x^\prime, \quad\forall u \in L^s(\mathds{R}^{N-1}),
 \end{equation}
 \begin{equation}\label{Property ii}
  \int_{\mathds{R}^{N-1}}|u^\ast-v^\ast|^s\mathrm{d}x^\prime\leq \int_{\mathds{R}^{N-1}}|u-v|^s\mathrm{d}x^\prime, \quad\forall u \in L^s(\mathds{R}^{N-1}),
\end{equation}   
and
\begin{equation}\label{Property iii}     
     \int_{\mathds{R}^{N-1}}|\nabla_{x^\prime} u^\ast|^2\mathrm{d}x^\prime\leq \int_{\mathds{R}^{N-1}}|\nabla_{x^\prime} u|^2\mathrm{d}x^\prime,\quad\forall u \in \mathcal{D}^{1,2}(\mathds{R}^N_+).
\end{equation}

	\begin{lemma}\label{Lemma-Simetria}Assume $(\rho_0)$ with $\gamma\geq 0$. Then, it holds
    \begin{equation}\label{Lp sym}
        \int_{\mathds{R}^N_+}|u^\ast-v^\ast|^p\mathrm{d}x\leq \int_{\mathds{R}^N_+}|u-v|^p\mathrm{d}x,\quad \text{$\forall u,v\in L^p(\mathds{R}^N_+)$}
    \end{equation}
and    
	\begin{equation}\label{symmetrization}
			\int_{\mathds{R}^N_+}\rho(x_N)|\nabla u^\ast|^2\mathrm{d}x\leq \int_{\mathds{R}^{N}_+}\rho(x_N)|\nabla u|^2\mathrm{d}x,\quad \forall u\in \mathcal{D}^{1,2}_\rho(\mathds{R}^N_+).
		\end{equation}
In particular, if $u \in \mathcal{D}^{1,2}_\rho(\mathds{R}^N_+)$, then $u^\ast \in \mathcal{D}^{1,2}_\rho(\mathds{R}^N_+)$.        
	\end{lemma}
	
	\begin{proof} \eqref{Lp sym} follows directly by \eqref{Property ii}. Let $u\in C^{\infty}_\delta(\mathds{R}^N_+)$. Using \eqref{Property ii}, for every $t\neq0$ we have  
		\begin{equation*}
			\int_{\mathds{R}^{N-1}}\left|\frac{u^\ast(x^\prime,x_N+t)-u^\ast(x^\prime,x_N)}{t}\right|^2\mathrm{d}x^\prime\leq \int_{\mathds{R}^{N-1}}\left|\frac{u(x^\prime,x_N+t)-u(x^\prime,x_N)}{t}\right|^2\mathrm{d}x^\prime.
		\end{equation*}
		Taking the limit as $t\longrightarrow0$, we obtain 
		\begin{equation}\label{uxN ineq}
			\int_{\mathds{R}^{N-1}}|u^\ast_{x_N}|^2\mathrm{d}x^\prime\leq \int_{\mathds{R}^{N-1}}|u_{x_N}|^2\mathrm{d}x^\prime. 
		\end{equation}
		Thus, by \eqref{Property iii} and \eqref{uxN ineq} 
		\begin{equation*}
			\int_{\mathds{R}^{N-1}}|\nabla u^\ast|^2\mathrm{d}x^\prime= \int_{\mathds{R}^{N-1}}|\nabla_{x^\prime}u^\ast|^2+|u^\ast_{x_N}|^2\mathrm{d}x^\prime\leq  \int_{\mathds{R}^{N-1}}|\nabla u|^2\mathrm{d}x^\prime.
		\end{equation*}
Multiplying the above inequality by $\rho(x_N)$ and integrating, we obtain \eqref{symmetrization}.	
Now, let $u\in \mathcal{D}^{1,2}_\rho(\mathds{R}^N_+)$ and $(u_n)\subset C^\infty_\delta(\mathds{R}^N_+)$ such that $u_n\longrightarrow u$ in $\mathcal{D}^{1,2}_\rho(\mathds{R}^N_+)$. Then, $\|u_n^\ast\|\leq \|u_n\|\leq C$ and up to a subsequence we have
 $u_n^\ast \rightharpoonup v$ in $\mathcal{D}^{1,2}_\rho(\mathds{R}^N_+)$, for some $v\in \mathcal{D}^{1,2}_\rho(\mathds{R}^N_+)$. Given that, by \eqref{Lp sym}, $\|u_n^\ast-u^\ast\|_{2^\ast,\mathds{R}^N_+}\leq \|u_n-u\|_{2^\ast, \mathds{R}^N_+}\longrightarrow 0$, then $v=u^\ast$ and $\|u^\ast\|\leq \liminf_{n\rightarrow \infty} \|u_n^\ast\|\leq \|u\|$, which conclude \eqref{symmetrization}.
\end{proof}

Building on the approach in \cite{PLLions}, we prove the following lemma:
\begin{lemma}\label{Radial Lemma} Assume $(\rho_0)$ with $\gamma\geq2$ and $u \in \mathcal{D}^{1,2}_\rho(\mathds{R}^N_+)$. Then, for $x^\prime\neq 0$ we have 
 
\begin{equation*}%\label{radial lemma}
			|u^\ast(x^\prime,x_N)|\leq C \frac{\|u^\ast\|}{|x^\prime|^{(N-1)/2}}.
		\end{equation*}	
\end{lemma}		
	
	\begin{proof} Fix $x^\prime \in \mathds{R}^{N-1}\setminus\{0\}$ and consider the function defined by 
		\begin{equation*}
			\omega(x_N):=\int_{\Gamma_{|x^\prime|}}u^\ast(y,x_N)dy,
		\end{equation*}
		where $\Gamma_{|x^\prime|}$ is the ball centered in $0$ and radius $|x^\prime|$ of $\mathds{R}^{N-1}$. Since $u^\ast$ is decreasing in $|x^\prime|$ we observe that 
		\begin{equation}\label{left inequality}
			\omega(x_N)\geq C|x^\prime|^{N-1}u^\ast(x^\prime,x_N).
		\end{equation}
		On the other hand, using the Holder inequality and that $\gamma\geq 2$, we get
	
        \begin{align*}
			|\omega(x_N)|\leq& \left(\int_{\Gamma_{|x^\prime|}}|u^\ast|^2\mathrm{d}y\right)^{1/2}\left(\int_{\Gamma_{|x^\prime|}}dy\right)^{1/2}\\
			\leq& C|x^\prime|^{(N-1)/2} \left(\int_{\mathds{R}^{N-1}}|u^\ast|^2dy\right)^{1/2}.
		\end{align*}
Thus, by \eqref{symmetrization}, we have that $u^\ast \in \mathcal{D}^{1,2}_\rho(\mathds{R}^N_+)$, and, using Lemma \ref{embedding} and assumption $(\rho_0)$, we obtain
	
        \begin{equation}\label{w inequality}
			\left(\int_{0}^{\infty}|\omega|^2\mathrm{d}x_N\right)^{1/2}\leq C|x^\prime|^{(N-1)/2} \left(\int_{\mathds{R}^N_+}|u^\ast|^2\mathrm{d}x\right)^{1/2}\leq  C|x^\prime|^{(N-1)/2}\|u^\ast\|.
		\end{equation}
Moreover,  
		\begin{align*}
			|\omega^\prime(x_N)|\leq&\int_{\Gamma_{|x^\prime|}}|u^\ast_{x_N}(y,x_N)|\mathrm{d}y\\
            \leq& \left(\int_{\Gamma_{|x^\prime|}}|u_{x_N}^\ast|^2\mathrm{d}x\right)^{1/2}\left(\int_{\Gamma_{|x^\prime|}}dy\right)^{1/2}\\
            \leq& C |x^\prime|^{(N-1)/2}\left(\int_{\mathds{R}^{N-1}}|\nabla u^\ast|^2\mathrm{d}x\right)^{1/2}\\
            \leq& C |x^\prime|^{(N-1)/2}\left(\int_{\mathds{R}^{N-1}}\rho(x_N)|\nabla u^\ast|^2\mathrm{d}x\right)^{1/2}
		\end{align*}
Thus,
        \begin{equation}\label{psi derivative}
\left(\int_{0}^{\infty}|\omega^\prime|^2\mathrm{d}x_N\right)^{1/2}\leq C|x^\prime|^{(N-1)/2}\|u^\ast\|.
		\end{equation}

		Now, observe that, by the definition of $\omega$ and Jensen's inequality (see \cite[Theorem 3.3]{Rudin}), for any $\varepsilon>0$, we have
        \begin{align*}
            |\omega(s_1)^2-\omega(s_2)^2|\leq& 2\int_{s_1}^{s_2}|\omega(s)\omega^\prime(s)|\mathrm{d}s\\
            \leq& C \left(\int_{\mathds{R}^{N-1}\times (s_1,s_2)}|u^\ast|^2\mathrm{d}x\right)^{1/2}\left(\int_{\mathds{R}^{N-1}\times (s_1,s_2)}|u^\ast_{x_N}|^2\mathrm{d}x\right)^{1/2}\\
            \leq& C\varepsilon,
        \end{align*}
for $s_1,s_2$ sufficiently large, since $u^\ast, u^\ast_{x_N}\in L^2(\mathds{R}^N_+)$. Thus, $\lim_{x_N\longrightarrow \infty}\omega(x_N)^2=L<\infty$, moreover $L=0$, since $\omega \in L^2(0,\infty)$.
Hence,        
		\[
		\omega(x_N)^2=w(x_N)^2-	\omega(\infty)^2=-2\int_{x_N}^{\infty}\omega(s)\omega^\prime(s)ds.
		\]
By Hölder inequality, \eqref{w inequality} and \eqref{psi derivative} we have
		\begin{equation}\label{ineq omega}
			\omega(x_N)^2\leq 2\|\omega\|_{2,\mathds{R}_+}\|\omega^\prime\|_{2,\mathds{R}_+}
			\leq C|x^\prime|^{(N-1)}\|u^\ast\|^2.
		\end{equation}
		Therefore, by \eqref{left inequality} and \eqref{ineq omega} we get
\begin{equation*}
    C|x^\prime|^{N-1}u^\ast(x^\prime,x_N)\leq \omega(x_N)\leq C|x^\prime|^{(N-1)/2}\|u^\ast\|
\end{equation*}
and thus
\[
|u^\ast(x^\prime,x_N)|\leq C \frac{\|u^\ast\|}{|x^\prime|^{(N-1)/2}}.
\]
\end{proof}

Now we consider the following lemma proved in \cite{PLLions}.
\begin{lemma}\label{Lions} Let $\Omega\subset \mathds{R}^m$ be a bounded domain with boundary Lipschitz, $2\leq l\in \mathds{Z}$. We write $H^1_S(\Omega\times\mathds{R}^l )=\{u\in H^1(\Omega\times \mathds{R}^l): \forall x \in \Omega,~u(x,y)=u(x,|y|)~~\text{$y\in \mathds{R}^l$}\}$ and $N=m+l$. The restriction to $H^1_S(\Omega\times \mathds{R}^l)$ of Sobolev embedding  $H^1(\Omega\times \mathds{R}^l)\hookrightarrow L^p(\Omega\times \mathds{R}^l)$ is compact if $p\in(2,2^\ast)$.
\end{lemma}

Exploring some ideas from \cite{Claudianor-Angelo} and employing an approach based on symmetrization arguments (see \cite{PLLions}), we prove the following compactness lemma:

\begin{proof}[Proof of Theorem \ref{Compact embedding}:] Let $(u_n)\subset \mathcal{R}^{1,2}_{\rho}(\mathds{R}^N_+)$ a bounded sequence. For some subsequence we have $u_n \rightharpoonup u$ in $\mathcal{R}^{1,2}_{\rho}(\mathds{R}^N_+)$. First, consider the interior case. 
Given $\varepsilon>0$, let $r_\varepsilon>0$ such that
\begin{equation*}
    \frac{1}{(1+x_N)^{\gamma-2}}<\varepsilon,\quad \forall x_N>r_\varepsilon.
\end{equation*}
Thus, by \eqref{Hardy p} and $(\rho_0)$,
\begin{equation*}
    \int_{\mathds{R}^{N-1}\times (r_\varepsilon,\infty)}|u_n-u|^2\mathrm{d}x\leq \varepsilon \int_{\mathds{R}^{N-1}\times (r_\varepsilon,\infty)}|u_n-u|^2(1+x_N)^{\gamma-2}\mathrm{d}x\leq C \varepsilon.
\end{equation*}
Considering $p\in(2,2^\ast)$, by interpolation inequality, for some $\alpha\in(0,1)$, we obtain
\begin{align*}
   \left(\int_{\mathds{R}^{N-1}\times (r_\varepsilon,\infty)}|u_n-u|^p\mathrm{d}x \right)^{1/p}\leq& \left(\int_{\mathds{R}^{N-1}\times (r_\varepsilon,\infty)}|u_n-u|^2\mathrm{d}x\right)^{\alpha/2}\\
   \times&\left(\int_{\mathds{R}^{N-1}\times (r_\varepsilon,\infty)}|u_n-u|^{2^\ast}\mathrm{d}x\right)^{(1-\alpha)/2^\ast}\\
   \leq& C\varepsilon^{\alpha/2}.
\end{align*}
On the other hand, by Lemma \ref{embedding}, we have
\begin{equation*}
    \mathcal{R}^{1,2}_\rho(\mathds{R}^{N-1}\times (0,r_\varepsilon))\hookrightarrow H^{1}_S(\mathds{R}^{N-1}\times (0,r_\varepsilon))\hookrightarrow L^p(\mathds{R}^{N-1}\times (0,r_\varepsilon)),
\end{equation*}
where the last embedding is compact due to Lemma \ref{Lions}, therefore the embedding $\mathcal{R}^{1,2}_\rho(\mathds{R}^N_+)\hookrightarrow L^p(\mathds{R}^N_+)$ is compact.

Now we focus on the trace embedding. By \eqref{symmetrization} we have $\|u_n^\ast\|$ bounded, then there exists $v\in \mathcal{R}^{1,2}_\rho(\mathds{R}^N_+)$ such that $u_n^\ast \rightharpoonup v$ in $\mathcal{R}^{1,2}_\rho(\mathds{R}^N_+)$, $u_n^\ast \longrightarrow v$ a.e. in $\mathds{R}^{N-1}$ and $u_n^\ast\longrightarrow v$ in $L^q_{\mathrm{loc}}(\mathds{R}^{N-1})$. Observe that $v=u^\ast$, given that, by \eqref{Lp sym} and the compactness in $L^p(\mathds{R}^N_+)$, we have $\|u_n^\ast-u^\ast\|_{p,\mathds{R}^N_+}\leq \|u_n-u\|_{p,\mathds{R}^N_+}\longrightarrow 0$. Therefore, we obtain
\begin{equation*}
    \int_{\Gamma_1}|u_n^\ast|^q\mathrm{d}x^\prime\longrightarrow \int_{\Gamma_1}|u^\ast|^q\mathrm{d}x^\prime.
\end{equation*}
Since $q>2$, due to Lemma \ref{Radial Lemma}, we have
\begin{equation*}
    |u_n^\ast|\leq C|x^\prime|^{-q(N-1)/2}\in L^1(\mathds{R}^{N-1}\setminus \Gamma_1).
\end{equation*}
Then,  by dominated convergence Theorem 
\begin{equation*}
    \int_{\mathds{R}^{N-1}\setminus\Gamma_1}|u_n^\ast|^q\mathrm{d}x^\prime\longrightarrow \int_{\mathds{R}^{N-1}\setminus\Gamma_1}|u^\ast|^q\mathrm{d}x^\prime.
\end{equation*}
Finally, by \eqref{Property i}, we obtain
\begin{equation*}
    \int_{\mathds{R}^{N-1}}|u_n|^q\mathrm{d}x^\prime=\int_{\mathds{R}^{N-1}}|u_n^\ast|^q\mathrm{d}x^\prime\longrightarrow \int_{\mathds{R}^{N-1}}|u^\ast|^q\mathrm{d}x^\prime=\int_{\mathds{R}^{N-1}}|u|^q\mathrm{d}x^\prime,
\end{equation*}
which concludes the theorem.
\end{proof}

\begin{proof}[Proof of Theorem \ref{Attainability}]
By Lemma~\ref{embedding}, we have that $S_{q,\mathds{R}^{N-1}} > 0$ when $\gamma > 1$ and $q \in [2, 2_\ast]$, and $S_{p,\mathds{R}^N_+} > 0$ when $\gamma \geq 2$ and $p \in [2, 2^\ast]$.

Now, suppose that $\gamma>2$, $q\in (2,2_\ast)$ and, without loss of generality, that $(u_n)$ is a minimizing sequence for \eqref{Constant1} with $\|u_n\|_{q,\mathds{R}^{N-1}} = 1$. By \eqref{Property i} and \eqref{Property iii}, the sequence $(u_n^\ast)$ is also a minimizing sequence with $\|u_n^\ast\|_{q,\mathds{R}^{N-1}} = 1$. Observe that $(u_n^\ast)$ is bounded in $\mathcal{D}^{1,2}_\rho(\mathds{R}^N_+)$, and, up to a subsequence, we have $u_n^\ast \rightharpoonup u$ in $\mathcal{D}^{1,2}_\rho(\mathds{R}^N_+)$. Since $\|u\| \leq \liminf_{n\rightarrow \infty} \|u_n\|$, it follows that $\|u\|^2 = S_{q,\mathds{R}^{N-1}}$. Moreover, by Lemma~\ref{Compact embedding}, we conclude that $\|u\|_{q,\mathds{R}^{N-1}} = 1$. 

As for $S_{p,\mathds{R}^N_+}$, using \eqref{Property i}, we obtain $\|u^\ast\|_{p,\mathds{R}^N_+} = \|u\|_{p,\mathds{R}^N_+}$. Therefore, applying the same argument, we conclude that $S_{p,\mathds{R}^N_+}$ is attained when $\gamma>2$ and $p\in [2,2^\ast]$.

We now prove that $S_{2,\mathds{R}^{N-1}}$ is not attained. The idea is inspired by the approach in \cite{Abreu-Furtado-Medeiros 2}. Given $u \in \mathcal{D}^{1,2}_\rho(\mathds{R}^N_+)$, define $u_\delta(x', x_N) := u(\delta x', x_N)$. Changing variables, we obtain:
\begin{equation*}
    S_{2,\mathds{R}^{N-1}} \int_{\mathds{R}^{N-1}} u^2 \, \mathrm{d}x' \leq \int_{\mathds{R}^N_+} \rho(x_N) \left( \delta^2 |\nabla_{x'} u|^2 + u_{x_N}^2 \right) \, \mathrm{d}x.
\end{equation*}
Letting $\delta \to 0$, we get:
\begin{equation}\label{uxN}
    S_{2,\mathds{R}^{N-1}} \int_{\mathds{R}^{N-1}} u^2 \, \mathrm{d}x' \leq \int_{\mathds{R}^N_+} \rho(x_N) u_{x_N}^2 \, \mathrm{d}x, \quad \forall u \in \mathcal{D}^{1,2}_\rho(\mathds{R}^N_+).
\end{equation}
Suppose that there exists $v \in \mathcal{D}^{1,2}_\rho(\mathds{R}^N_+)$ such that
\begin{equation*}
    S_{2,\mathds{R}^{N-1}} \int_{\mathds{R}^{N-1}} v^2 \, \mathrm{d}x' = \int_{\mathds{R}^N_+} \rho(x_N) |\nabla v|^2 \, \mathrm{d}x.
\end{equation*}
Then, combining this with \eqref{uxN}, we obtain $\nabla_{x'} v = 0$, so that $v(x', x_N) = w(x_N)$ for some function $w$. Consequently, $\|v\|_{2, \mathds{R}^{N-1}}^2 = \infty$, which is a contradiction, since $\mathcal{D}^{1,2}_\rho(\mathds{R}^N_+) \hookrightarrow L^2(\mathds{R}^{N-1})$ by Lemma \ref{embedding}. The same argument shows that $S_{2,\mathds{R}^N_+}$ cannot be attained.
\end{proof}

%%%%%%%%%%%%%%%%%%%%%%%%%%%%%%%%%%%%%%%%%%%%%%%%%%%%%%%%%%%%%%%%%%%%%%%%%%%%%%%%%%
%
%                  REGULARITY RESULTS
%
%%%%%%%%%%%%%%%%%%%%%%%%%%%%%%%%%%%%%%%%%%%%%%%%%%%%%%%%%%%%%%%%%%%%%%%%%%%%%%%%%%
\section{Regularity results}

The main objective of this section is to proof the Theorems \ref{Holder regularity result} and \ref{H2_loc regularity theorem}.

\subsection{Hölder Regularity}

In this section, we investigate the Hölder regularity of weak solutions of problem \eqref{PG}. With an approach inspired by techniques found in \cite{Abreu-JM-Medeiros, Barrios, Drabek}, we provide the boundedness of the weak solutions, allowing us to reach the boundary regularity desired for the case where $a$ and $b$ have inverse signs. The interior regularity for $a,b>0$ in the subcritical case is established by verifying that $u^r$ belongs to $\mathcal{D}^{1,2}_\rho(\mathds{R}^N_+)$ and applying the classical regularity theory.

\begin{lemma}\label{u^r belongs D12} Assume $(\rho_0)$ with $\gamma\geq 2$, $p\in[2,2^\ast)$ and $q\in[2,2_\ast)$. If $u \in \mathcal{D}^{1,2}_\rho(\mathds{R}^N_+)$ is a nonnegative weak solution for \eqref{PG}, then $u^r \in \mathcal{D}^{1,2}_\rho(\mathds{R}^N_+)$ for all $r>1$.
    
\end{lemma}

\begin{proof} Let $1<r_1=\min\{2^\ast-p+1, 2_\ast-q+1\}$ and define $T_ku= \min\{u^{r_1}, k^{r_1}\}$, with $k \in \mathds{N}$. Observe that $T_ku \in \mathcal{D}^{1,2}_\rho(\mathds{R}^N_+)$ with 
\begin{equation*}
    \nabla T_ku=\left\{
		\begin{aligned}
			(r_1-1)u^{r_1-1}\nabla u\ &\mbox{if }&\ u\leq k,
			\vspace{0.2cm}\\
			0\ &\mbox{if }&\ u>k.
		\end{aligned}
		\right.
\end{equation*}
From \eqref{weak solution} and a density argument we have
\begin{equation*}
    \int_{\mathds{R}^N_+}\rho(x_N)\nabla u \nabla(T_ku)\mathrm{d}x= a\int_{\mathds{R}^N_+}u^{p-1} T_k(u)\mathrm{d}x+b\int_{\mathds{R}^{N-1}}u^{q-1}T_k(u)\mathrm{d}x^\prime,
\end{equation*}
then
\begin{equation*}
    (r_1-1)\int_{\{u\leq k\}}\rho(x_N)|\nabla u|^2u^{r_1-1}\mathrm{d}x\leq 2|a| \int_{\mathds{R}^N_+} u^{p-1+r_1}\mathrm{d}x+2|b|\int_{\mathds{R}^{N-1}}u^{q-1+r_1}\mathrm{d}x^\prime
\end{equation*}
and
\begin{equation*}
    \int_{\{u\leq k\}}\rho(x_N)|\nabla (u^{\frac{r_1+1}{2}})|^2\mathrm{d}x\leq |a| \int_{\mathds{R}^N_+} u^{p-1+r_1}\mathrm{d}x+|b|\int_{\mathds{R}^{N-1}}u^{q-1+r_1}\mathrm{d}x^\prime.
\end{equation*}
The right side is finite by restriction to $r_1$. Thus, passing to the limit when $k \longrightarrow \infty$ and applying the monotone convergence theorem, we conclude that $u^{\frac{r_1+1}{2}} \in \mathcal{D}^{1,2}_\rho(\mathds{R}^N_+) $. Now, define $s_1:= (r_1+1)/2$ and observe that by Lemma \ref{embedding} we have
\begin{equation}\label{u in Ls}
   u \in \left\{
		\begin{aligned}
			L^s(\mathds{R}^N_+) \ &\mbox{for }&\ s \in\ &[2, 2^\ast s_1],
			\vspace{0.2cm}\\
			L^s(\mathds{R}^{N-1}) \ &\mbox{for }&\ s \in\ & [2, 2_\ast s_1].
		\end{aligned}
		\right.
\end{equation}
Now, let $r_2 = \min\{2^\ast s_1 -p+1, 2_\ast s_1-q+1\}$ and take $T_k(u)=\min\{u^{r_2},k^{r_2}\}$. Taking $T_k(u)$ as test function and arguing as in the previous case, we obtain
\[
\int_{\mathds{R}^N_+}\rho(x_N)|\nabla (u^{\frac{r_2+1}{2}})|^2\mathrm{d}x\leq |a| \int_{\mathds{R}^N_+} u^{p-1+r_2}\mathrm{d}x+|b|\int_{\mathds{R}^{N-1}}u^{q-1+r_2}\mathrm{d}x^\prime,
\]
and thus, by \eqref{u in Ls}, $u^{s_2} \in \mathcal{D}^{1,2}_\rho(\mathds{R}^N_+)$ for $s_2= (r_2+1)/2$. We can iterate this argument and we obtain a sequence $(r_n)$ such that
\[
r_n= \min\{2^\ast s_{n-1}+1-p,2_\ast s_{n-1}+1-q\}\quad\text{and}\quad u^{s_n}\in \mathcal{D}^{1,2}_\rho(\mathds{R}^N_+),
\]
where $s_n= (r_n+1)/2$. Clearly we have $r_n\longrightarrow \infty$. In fact, suppose that for some $M>0$, $r_n\leq M$. Going if necessary to a subsequence, we have $r_n\longrightarrow r\geq 1$. Then, passing to the limit in the definition of $r_n$, we get
\begin{equation*}
    \frac{p-2}{2^\ast-2}=\frac{r+1}{2}\quad\text{or}\quad \frac{q-2}{2_\ast-2}=\frac{r+1}{2}.
\end{equation*}
Which is a contradiction in both cases, since $p,q$ are subcritical and $r\geq1$.

Now, observe that 
\begin{equation}\label{u belongs Ls}
u \in L^s(\mathds{R}^N_+)\cap L^s(\mathds{R}^{N-1})\quad \forall s>2.
\end{equation}
In fact, given $s>2$, we can choose $n$ sufficient large to $2s_n>s$, then, since $u^{s_n}$ belongs to $ \mathcal{D}^{1,2}_\rho(\mathds{R}^N_+)$ and hence to $L^{2^\ast s_n}(\mathds{R}^N_+)$, by interpolation inequality we have, for some $\alpha \in (0,1)$,
\[
\|u\|_{s, \mathds{R}^N_+}\leq \|u\|_{2, \mathds{R}^N_+}^{1-\alpha}\|u\|_{2^\ast s_n, \mathds{R}^N_+}^\alpha <\infty.
\]
Finally, we claim that $u^r \in \mathcal{D}^{1,2}_\rho(\mathds{R}^N_+)$ for all $r>1$. Indeed, given $r>1$, let $s>1$ such that $(s+1)/2=r$. Then, reproducing the initial argument, we get
\[
\int_{\mathds{R}^N_+}\rho(x_N)|\nabla (u^{\frac{s+1}{2}})|^2\mathrm{d}x\leq a \int_{\mathds{R}^N_+} |u|^{p-1+s}\mathrm{d}x+b\int_{\mathds{R}^{N-1}}|u|^{q-1+s}\mathrm{d}x^\prime,
\]
where the right side is finite by \eqref{u belongs Ls}, which concludes the result.
\end{proof}

\begin{lemma}\label{Linfty}
Assume that condition $(\rho_0)$ holds. Then, in each of following assertions:
\begin{itemize}
    \item[$(i)$]  $u$ is a nonnegative weak solution in $E_q(\mathds{R}^{N-1})$, with $\gamma\geq 2$, $a > 0$, $b \leq 0$, $p \in [2,2^\ast)$, and $q \in (1,\infty)$;
    \item[$(ii)$]  $u$ is a nonnegative weak solution in $E_p(\mathds{R}^N_+)$, with $\gamma>1$, $a \leq 0$, $b > 0$, $p \in (1,\infty)$, and $q \in [2, 2_\ast)$,
\end{itemize}
we have $u \in L^\infty(\mathds{R}^N_+) \cap L^\infty(\mathds{R}^{N-1})$.
\end{lemma}

\begin{proof} We begin by considering case $(i)$. Since $\gamma\geq 2$, by Lemma \ref{embedding} we have $E_q(\mathds{R}^{N-1})\hookrightarrow \mathcal{D}^{1,2}_\rho(\mathds{R}^N_+)\hookrightarrow L^p(\mathds{R}^N_+)$ for $p\in[2,2^\ast]$. Then, arguing by density, the definition of weak solution holds for testing functions in $E_q(\mathds{R}^{N-1})$. For $m\in \mathds{N}$, define $u_m=\min\{u,m\}$. Taking $\varphi=u_m^{2k+1}$ as a testing function in \eqref{weak solution}, with $k>0$ we get
\begin{align*}
    \int_{\mathds{R}^N_+} \rho(x_N)\nabla u\nabla (u_m^{2k+1})\mathrm{d}x&=a\int_{\mathds{R}^N_+}u^{p-1}u_m^{2k+1}\mathrm{d}x+b\int_{\mathds{R}^{N-1}}u^{q-1}u_m^{2k+1}\mathrm{d}x^\prime\\
    &\leq a\int_{\mathds{R}^{N}_+}u^{2k+p}\mathrm{d}x,
\end{align*}
where the inequality follows from the nonnegativity of \( u \), the assumption \( b \leq 0 \), and the fact that \( u_m \leq u \). Given that
\begin{equation*}
C\left(\int_{\mathds{R}^N_+}u^{2^\ast}\mathrm{d}x\right)^{2/2^\ast}\leq \int_{\mathds{R}^N_+}|\nabla u|^2\mathrm{d}x\leq  \int_{\mathds{R}^N_+}\rho(x_N)|\nabla u|^2\mathrm{d}x,
\end{equation*}
we obtain
\begin{align*}
    \int_{\mathds{R}^N_+} \rho(x_N) \nabla u\nabla (u_m^{2k+1})\mathrm{d}x=& \frac{2k+1}{(k+1)^2}\int_{\mathds{R}^{N_+}}\rho(x_N)|\nabla (u_m^{k+1})|^2\mathrm{d}x\\
    \geq& \,C \frac{2k+1}{(k+1)^2}\left(\int_{\mathds{R}^{N}_+}u_m^{(k+1)2^\ast}\mathrm{d}x\right)^{2/2^\ast}.
\end{align*}
Hence, we deduce that
\begin{equation*}
    C \frac{2k+1}{(k+1)^2}\left(\int_{\mathds{R}^{N}_+}u_m^{(k+1) 2^\ast}\mathrm{d}x \right)^{2/2^\ast}\leq \int_{\mathds{R}^N_+}u^{2k+p}\mathrm{d}x =\int_{\mathds{R}^N_+}u^{2(k+1)}u^{p-2}\mathrm{d}x .
\end{equation*}
Let $\zeta \in [2,2^\ast)$ such that
$$
\frac{2}{\zeta}+\frac{p-2}{2^*}=1.
$$
Applying the Hölder inequality, we obtain 
\begin{equation*}
   \left(\int_{\mathds{R}^N_+}u_m^{(k+1) 2^\ast}\mathrm{d}x \right)^{1/ 2^\ast}\leq C_1\frac{k+1}{(2k+1)^{1/2}} \left(\int_{\mathds{R}^N_+}u^{(k+1)\zeta}\mathrm{d}x \right)^{1/\zeta}.
\end{equation*}
For simplicity, we denote $\|\cdot\|_{s,\mathds{R}^N_+}=\|\cdot\|_s$. Then, letting $m\longrightarrow \infty$, we obtain
\begin{equation}\label{recursive}
    \|u\|_{(k+1) 2^\ast}\leq C_1^{1/(k+1)}\left[\frac{k+1}{(2k+1)^{1/2}}\right]^{1/(k+1)}\|u\|_{(k+1)\zeta},\quad \forall k>0.
\end{equation}
Next, we choose $k=k_1>0$ such that $(k_1+1)\zeta=2^\ast$ and define
$$
(k_{n}+1)\zeta=(k_{n-1}+1)2^\ast, \quad n\geq2.
$$
Notice that 
\begin{equation}\label{kn}
k_n=\left(\frac{2^\ast}{\zeta}\right)^n-1 \quad \text{and}\quad k_n \longrightarrow \infty,
\end{equation}
since $\zeta \in [2, 2^\ast)$, and 
\begin{equation}\label{n induction}
     \|u\|_{(k_n+1) 2^\ast}\leq C_1^{1/(k_n+1)}\left[\frac{k_n+1}{(2k_n+1)^{1/2}}\right]^{1/(k_n+1)}\|u\|_{(k_{n-1}+1) 2^\ast}.
\end{equation}
Thus, using \eqref{n induction} and \eqref{recursive} recursively, it follows that
\begin{equation*}
     \|u\|_{(k_n+1) 2^\ast}\leq C_1^{\sum_{i=1}^n1/(k_i+1)}\prod_{i=1}^n\left[\frac{k_i+1}{(2k_i+1)^{1/2}}\right]^{1/(k_i+1)}\|u\|_{ 2^\ast}.
\end{equation*}
Write
\begin{equation*}
    \left[\frac{k_i+1}{(2k_i+1)^{1/2}}\right]^{1/(k_i+1)}=\left[\left(\frac{k_i+1}{(2k_i+1)^{1/2}}\right)^{1/\sqrt{k_i+1}}\right]^{1/\sqrt{k_i+1}}.
\end{equation*}
Observe that $\left(\frac{k+1}{(2k+1)^{1/2}}\right)^{1/\sqrt{k+1}}>1$ for $k>0$ and $\displaystyle\lim_{k\longrightarrow \infty}\left(\frac{k+1}{(2k+1)^{1/2}}\right)^{1/\sqrt{k+1}}=1$, then there exists $C_2>1$ such that
\begin{equation*}
     \|u\|_{(k_n+1) 2^\ast}\leq C_1^{\sum_{i=1}^n\frac{1} {k_i+1}}C_2^{\sum_{i=1}^n\frac{1}{\sqrt{k_i+1}}}\|u\|_{ 2^\ast}.
\end{equation*}
Now, given that $\zeta\in [2, 2^\ast)$, by \eqref{kn} we have
\begin{equation*}
    \sum_{n=1}^\infty \frac{1}{k_n+1}= \sum_{n=1}^\infty \left(\frac{\zeta}{ 2^\ast}\right)^n<\infty \quad\text{and}\quad \sum_{n=1}^\infty \frac{1}{\sqrt{k_n+1}}=\sum_{n=1}^\infty \left(\sqrt{\frac{\zeta}{ 2^\ast}}\right)^n<\infty.
\end{equation*}
Therefore, taking the limit as $n\longrightarrow \infty$, we obtain
\begin{equation*}
    \|u\|_{\infty}\leq C \|u\|_{ 2^\ast}
\end{equation*}
and hence $u \in L^\infty(\mathds{R}^N_+)$. Set 
\begin{equation}\label{Omega k e Gamma k}
\Omega_k:=\{x\in \overline{\mathds{R}^N_+}: u(x)>k\}\quad\text{and}\quad \Gamma_k:=\{x\in \mathds{R}^{N-1}:u(x)>k\}
\end{equation}
and note that $|\Omega_k|<\infty$, since $u\in L^{2^\ast}(\mathds{R}^N_+)\cap L^2(\mathds{R}^{N-1})$. Choosing $k$ such that $\|u\|_\infty<k$ and taking $T_ku$, given by
\begin{equation*}
    T_ku=\left\lbrace
		\begin{aligned}
			&u-k,\quad&\mbox{in }&\,\Omega_k,
			\vspace{0.2cm}\\
		&0,\quad &\mbox{in }&\,\overline{\mathds{R}^N_+}\setminus\Omega_k,
		\end{aligned}
		\right. 
\end{equation*}
as a testing function, we get
\begin{equation*}
    0=b\int_{\Gamma_k}u^{q-1}T_ku\mathrm{d}x^\prime.
\end{equation*}
Given that $u^{q-1}$ and $T_ku$ are positive in $\Gamma_k$, it remains that $|\Gamma_k|=0$ and $u \in L^\infty(\mathds{R}^{N-1})$, which concludes the first item. 

For item $(ii)$, we can apply a bootstrap argument similar to the previous case, which yields $u \in L^\infty(\mathds{R}^{N-1})$ with $\|u\|_{\infty, \mathds{R}^{N-1}}\leq C\|u\|_{2_\ast, \mathds{R}^{N-1}}$. Finally, choosing $k$ such that $\|u\|_{\infty,\mathds{R}^{N-1}}<k$ and taking again the truncation function $T_ku$ as a testing function, we obtain
\begin{equation*}
    0\leq \int_{\Omega_k}\rho(x_N)|\nabla u|^2\mathrm{d}x=a\int_{\mathds{R}^N_+}u^{p-1}T_ku\mathrm{d}x+b\int_{\mathds{R}^{N-1}}u^{q-1}T_ku\mathrm{d}x^\prime\leq 0,
\end{equation*}
hence, either $u$ is constant or $u\leq k$ a.e. in $\mathds{R}^N_+$ and we have $u \in L^\infty(\mathds{R}^N_+)$, concluding the result.
\end{proof}

\begin{lemma}\label{Linfty 2}
Assume condition $(\rho_0)$ holds with $\gamma\geq 0$. Then, in each of the following cases:
\begin{itemize}
\item[$(i)$] $u$ is a nonnegative weak solution in $E_q(\mathds{R}^{N-1})$, with $a > 0$, $b \leq 0$, $p=2^\ast$ and $q\in(1,\infty)$;
\item[$(ii)$] $u$ is a nonnegative weak solution in $E_p(\mathds{R}^N_+)$, with $a \leq 0$, and $b > 0$, $p\in(1,\infty)$ and $q=2_\ast$,
\end{itemize}
it follows that $u \in L^\infty(\mathds{R}^N_+) \cap L^\infty(\mathds{R}^{N-1})$.
\end{lemma}

\begin{proof} Let us prove the first item. By the classical Sobolev inequality together with assumption $(\rho_0)$, we also have $E_q(\mathds{R}^{N-1})\hookrightarrow \mathcal{D}^{1,2}_\rho(\mathds{R}^N_+)\hookrightarrow L^{2^\ast}(\mathds{R}^N_+)$. Then, arguing by density the definition of weak solution \eqref{weak solution} holds for testing functions in $E_q(\mathds{R}^{N-1})$. Let $\varphi:[0,\infty)\longrightarrow \mathds{R}$ defined by 
\begin{equation*}
    \varphi(t)= \varphi_{m,k}(t)=\left\lbrace
		\begin{aligned}
			&t^k,\quad&\mbox{for }&\,t\leq m\\
		&km^{k-1}(t-m)+m^k,\quad &\mbox{for}&\,t\geq m,
		\end{aligned}
		\right. 
\end{equation*}
for $k>1$ and $m>0$.
Then we have
\begin{equation*}
    \varphi^\prime(t)=\left\lbrace
		\begin{aligned}
			&kt^{k-1},\quad&\mbox{for }&\,t\leq m\\
		&km^{k-1},\quad &\mbox{for}&\,t\geq m.
		\end{aligned}
		\right. 
\end{equation*}
Take $\varphi(u)\varphi'(u)$ as a testing function in \eqref{weak solution}, then, given that $\varphi(u)\varphi'(u)\geq0$ and $b\leq 0$, we obtain 
\begin{align}\label{first inequality}
   \nonumber \int_{\mathds{R}^{N}_+}\rho(x_N)\nabla u \nabla (\varphi(u)\varphi^\prime(u)) \, \mathrm{d} x&=a\int_{\mathds{R}^{N}_+}u^{2^\ast-1}\varphi(u)\varphi^\prime(u) \mathrm{d}x  +b\int_{\mathds{R}^{N-1}}u^{q-1} \varphi(u)\varphi^\prime(u) \, \mathrm{d} x\\
    &\leq a\int_{\mathds{R}^{N}_+}u^{2^\ast-1} \varphi(u)\varphi^\prime(u) \, \mathrm{d} x.
\end{align}
Now, observe that
\begin{equation*}
    \nabla \varphi(u)=\left\lbrace
		\begin{aligned}
			&ku^{k-1}\nabla u,\quad&\mbox{for }&\,u\leq m\\
		&km^{k-1}\nabla u,\quad &\mbox{for}&\,u\geq m
		\end{aligned}
		\right. 
\end{equation*}
and
\begin{equation*}
    \nabla (\varphi(u)\varphi^\prime(u))=\left\lbrace
		\begin{aligned}
			&k(2k-1)u^{(2k-1)}\nabla u,\quad&\mbox{for }&\,u\leq m\\
		&k^2 m^{2(k-1)}\nabla u,\quad &\mbox{for }&\,u\geq m.
		\end{aligned}
		\right. 
\end{equation*}
Then $\nabla u \nabla (\varphi(u)\varphi^\prime(u))\geq |\nabla \varphi(u)|^2$ and
\begin{equation}\label{LHS}
    \int_{\mathds{R}^N_+}\rho(x_N)\nabla u\nabla(\varphi(u)\varphi^\prime(u))\mathrm{d}x\geq \int_{\mathds{R}^N_+}\rho(x_N)|\nabla \varphi(u)|^2\mathrm{d}x\geq C\left(\int_{\mathds{R}^{N}_+}\varphi(u)^{2^\ast}\mathrm{d}x\right)^{2/2^\ast}.
\end{equation}
On the other hand, since $u\varphi^\prime(u)\leq k\varphi(u)$, we have
\begin{equation*}
    \int_{\mathds{R}^{N}_+}u^{2^\ast-1}\varphi(u)\varphi^\prime(u)\mathrm{d}x\leq k\int_{\mathds{R}^{N}_+}\varphi(u)^2u^{2^\ast-2}\mathrm{d}x.
\end{equation*}
Thus, applying the inequality above and \eqref{LHS} in \eqref{first inequality} we obtain
\begin{equation}\label{varphi inequality}
    \left(\int_{\mathds{R}^N_+}\varphi(u)^{2^\ast}\mathrm{d}x\right)^{2/2^\ast}\leq C_1k\int_{\mathds{R}^N_+}\varphi(u)^2u^{2^\ast-2}\mathrm{d}x, \quad\forall k>1.
\end{equation}

\textbf{Claim:} Let $k_1$ such that $2k_1=2^\ast$. Then $u\in L^{k_12^\ast}(\mathds{R}^N_+)$.

Let $R>0$ to be determined later. By Holder's inequality we obtain
\begin{align*}
    \int_{\mathds{R}^N_+}\varphi(u)^2u^{2^\ast-2}\mathrm{d}x=&\int_{\{u\leq R\}}\varphi(u)u^{2^\ast-2}\mathrm{d}x+\int_{\{u>R\}}\varphi(u)^2u^{2^\ast-2}\mathrm{d}x\\
    \leq& R^{2^\ast-2}\int_{\mathds{R}^N_+}\varphi(u)^2\mathrm{d}x+\left(\int_{\mathds{R}^N_+}\varphi(u)^{2^\ast}\mathrm{d}x\right)^{2/2^\ast}\left(\int_{\{u>R\}}u^{2^\ast}\mathrm{d}x\right)^{(2^\ast-2)/2^\ast}.
\end{align*}
By dominated convergence Theorem we can choose $R>0$ such that 
\begin{equation*}
    \left(\int_{\{u>R\}}u^{2^\ast}\mathrm{d}x\right)^{(2^\ast-2)/2^\ast} \leq \frac{1}{2C_1k_1}.
\end{equation*}
Thus,
\begin{equation*}
    \int_{\mathds{R}^N_+}\varphi(u)^2u^{2^\ast-2}\mathrm{d}x\leq R^{2^\ast-2}\int_{\mathds{R}^N_+}\varphi(u)^2\mathrm{d}x+\frac{1}{2C_1k_1}\left(\int_{\mathds{R}^N_+}\varphi(u)^{2^\ast}\mathrm{d}x\right)^{2/2^\ast}.
\end{equation*}
Applying this inequality in \eqref{varphi inequality} we obtain
\begin{equation*}
    \left(\int_{\mathds{R}^N_+}\varphi(u)^{2^\ast}\mathrm{d}x\right)^{2/2^\ast}\leq C_1k_1 R^{2^\ast-2}\int_{\mathds{R}^N_+}\varphi(u)^2\mathrm{d}x.
\end{equation*}
Hence, given that $\varphi_{m,k}(u)\leq u^{k}$, letting $m\longrightarrow\infty$, by dominated convergence Theorem we obtain
\begin{equation*}
     \left(\int_{\mathds{R}^N_+}u^{k_12^\ast}\mathrm{d}x\right)^{2/2^\ast} \leq C_1k_1R^{2^\ast-2}\int_{\mathds{R}^N_+}u^{2^\ast}\mathrm{d}x<\infty
\end{equation*}
and we conclude the claim.

Applying dominated convergence Theorem again in \eqref{varphi inequality} we obtain
\begin{equation*}
    \left(\int_{\mathds{R}^N_+}u^{k2^\ast}\mathrm{d}x\right)^{2/2^\ast}\leq C_1 k \int_{\mathds{R}^N_+}u^{2k+2^\ast-2}\mathrm{d}x,
\end{equation*}
and consequently
\begin{equation}\label{bootstrap}
    \left(\int_{\mathds{R}^N_+}u^{k2^\ast}\mathrm{d}x\right)^{1/2^\ast(k-1)}\leq (C_1k)^{1/2(k-1)} \left(\int_{\mathds{R}^N_+} u^{2k
    +2^\ast-2}\mathrm{d}x\right)^{1/2(k-1)},\quad \forall k>1.
\end{equation}
Define $2k_{n}+2^\ast-2=2^\ast k_{n-1}$, then
\begin{equation*}
k_n-1=\frac{2^\ast}{2}(k_{n-1}-1)=\left(\frac{2^\ast}{2}\right)^{n-1}(k_1-1)
\end{equation*}
and $k_n\longrightarrow \infty$. Therefore, we have
\begin{equation*}
    \left(\int_{\mathds{R}^N_+}u^{k_n2^\ast}\mathrm{d}x\right)^{1/2^\ast(k_n-1)}\leq (C_1k_n)^{1/2(k_n-1)} \left(\int_{\mathds{R}^N_+} u^{k_{n-1}2^\ast}\mathrm{d}x\right)^{1/2^\ast(k_{n-1}-1)},
\end{equation*}
and applying recursively the inequality \eqref{bootstrap} we obtain
\begin{equation*}
    \left(\int_{\mathds{R}^N_+}u^{k_n2^\ast}\mathrm{d}x\right)^{1/2^\ast(k_n-1)}\leq \prod_{i=2}^n(C_1k_i)^{1/2(k_i-1)} \left(\int_{\mathds{R}^N_+} u^{k_{1}2^\ast}\mathrm{d}x\right)^{1/2^\ast(k_{1}-1)}.
\end{equation*}
Write $k_i^{1/2(k_i-1)}=\left[k_i^{1/\sqrt{k_i+1}}\right]^{1/2\sqrt{k_i+1}}$ and note that $\displaystyle\lim_{k\longrightarrow \infty}k^{1/\sqrt{k-1}}=1$, then there exists $C_2>1$ such that
\begin{equation*}
  \left(\int_{\mathds{R}^N_+}u^{k_n2^\ast}\mathrm{d}x\right)^{1/2^\ast(k_n-1)}\leq C_1^{\sum_{i=2}^n{\frac{1}{2(k_i-1)}}} C_2^{\sum_{i=2}^n{\frac{1}{2\sqrt{k_i-1}}}} \left(\int_{\mathds{R}^N_+} u^{k_{1}2^\ast}\mathrm{d}x\right)^{1/2^\ast(k_{1}-1)}.
  \end{equation*}
Since
\begin{equation*}
    \sum_{n=2}^\infty \frac{1}{2(k_n-1)}=\frac{1}{2(k_1-1)} \sum_{n=2}^\infty \left(\frac{2}{ 2^\ast}\right)^{n-1}<\infty  
\end{equation*}
and
\begin{equation*}
    \sum_{n=2}^\infty \frac{1}{2\sqrt{k_n-1}}=\frac{1}{2\sqrt{k_1-1}}\sum_{n=2}^\infty \left(\sqrt{\frac{2}{ 2^\ast}}\right)^{n-1}<\infty,
\end{equation*}
we obtain
\begin{equation*}
    \left(\int_{\mathds{R}^N_+}u^{k_n2^\ast}\mathrm{d}x\right)^{1/2^\ast(k_n-1)}\leq C_0\|u\|_{k_12^\ast,\mathds{R}^N_+}^{\frac{k_1}{k_1-1}}.
\end{equation*}
Therefore, letting $n\longrightarrow \infty$ we obtain $u\in L^{\infty}(\mathds{R}^N_+)$ with $\|u\|_{\infty, \mathds{R}^N_+}\leq C_0 \|u\|_{k_12^\ast,\mathds{R}^N_+}^{\frac{k_1}{k_1-1}}$. Finally, arguing as in Lemma \ref{Linfty}, we obtain $u\in L^\infty(\mathds{R}^{N-1})$ and conclude the first item.

For the second item, a similar argument shows that $u \in L^\infty(\mathds{R}^{N-1})$, with the bound
\begin{equation*}
\|u\|_{\infty,\mathds{R}^{N-1}}\leq C_0 \|u\|_{k_12_\ast,\mathds{R}^{N-1}}^{\frac{k_1}{k_1-1}}.
\end{equation*}
Finally, following the argument at the end of the second item in Lemma~\ref{Linfty}, we conclude that $u\in L^\infty(\mathds{R}^N_+)$ and complete case $(ii)$. 
\end{proof}

\begin{lemma}\label{Lib} Assume $(\rho_0)$ with $\rho\in C^1[0,\infty)$. Then, if $u$ is a bounded weak solution of \eqref{PG}, then $u \in C^{1,\alpha}_{\mathrm{loc}}(\overline{\mathds{R}^N_+})$.
 \end{lemma}
 \begin{proof} The proof follows by application of \cite[Theorem 2]{Lieberman}.
\end{proof}

Now we follow with the proof of Theorem \ref{Holder regularity result}

\begin{proof}[Proof of Theorem \ref{Holder regularity result}:] Let $u \in \mathcal{D}^{1,2}_\rho(\mathds{R}^N_+)$ be a nonnegative weak solution of \eqref{PG}. In any of the cases $(i)-(iv)$, applying the Lemmas \ref{Linfty}, \ref{Linfty 2} and \ref{Lib} we obtain that $u \in C^{1,\alpha}_{\mathrm{loc}}(\overline{\mathds{R}^N_+})$ for some $\alpha\in (0,1)$. Furthermore, since $u^{p-1}$ is bounded, by the estimates in \cite{ADN 1} we have $u \in W^{2,r}_{\mathrm{loc}}(\mathds{R}^N_+)$ for any $r>1$.

Applying the Harnack inequality from \cite[Theorem 1.1]{Trudinger}, we conclude that $u>0$ in $\mathds{R}^N_+$. Furthermore, using the inequality $|s^{p-1}-t^{p-1}|\leq M|s-t|$ that holds locally in $(0,\infty)$, we deduce that locally $|u(x)^{p-1}-u(y)^{p-1}|\leq C|x-y|^\alpha$, so that $u^{p-1}\in C^{0,\alpha}_{\mathrm{loc}}(\mathds{R}^N_+)$. Since $\rho\in C^{1,\alpha}_{\mathrm{loc}}[0,\infty)$, by \cite[Theorem 9.19]{GT}, we conclude that $u \in C^{2,\alpha}_{\mathrm{loc}} (\mathds{R}^N_+)$. Then, by the Hopf lemma (see \cite[Lemma 3.4]{GT}) we have $u > 0$ in $\overline{\mathds{R}^N_+}$ and then $u^{p-1} \in C^{0,\alpha}_{\mathrm{loc}}(\overline{\mathds{R}^N_+})$. Therefore, applying \cite[Theorem 9.19]{GT} once more, we obtain $u \in C^{2,\alpha}_{\mathrm{loc}}(\overline{\mathds{R}^N_+})$.

Finally, in the case where $\gamma\geq 2$ and $a,b>0$, let $s>2$ such that $s>N/2$. By Lemma \ref{u^r belongs D12}, we have $u^{\frac{s(p-1)}{2}}\in \mathcal{D}^{1,2}_\rho(\mathds{R}^N_+)$. Moreover, since $\gamma\geq 2$, it follows from Lemma \ref{embedding} that $u^{\frac{s(p-1)}{2}}\in L^2(\mathds{R}^N_+)$. Hence, $u^{p-1}\in L^s(\mathds{R}^N_+)$, and by the estimates in \cite{ADN 1}, we deduce that $u \in W^{2,s}_{\mathrm{loc}}(\mathds{R}^N_+)$ and then $u\in C^{1,\alpha}_{\mathrm{loc}} (\mathds{R}^N_+)$ for some $\alpha\in(0,1)$. Then, arguing as in the previous cases we obtain $u \in C^{2,\alpha}_{\mathrm{loc}} (\mathds{R}^N_+)$ and $u>0$ in $\mathds{R}^N_+$.
\end{proof}

\subsection{$H^2_{\mathrm{loc}}$-Regularity} In this part, we employ the method of difference quotients to establish $H^2_{\mathrm{loc}}(\overline{\mathds{R}^N_+})$ regularity.

The following lemma is standard, but we include a proof for the sake of completeness and clarity.

\begin{lemma}\label{limit dq}If $u \in \mathcal{D}^{1,2}(\mathds{R}^N_+)$, then, for $i=1,\cdots,N-1$, $D_h^i(u) \longrightarrow \frac{\partial u}{\partial x_i}$ in $L^2(\mathds{R}^N_+)$ as $h \longrightarrow 0$, where $D_h^i(u)$ is a differential quotient given by
\begin{equation*}
D_h^i(u)=\frac{u(x+he_i)-u(x)}{h}.
\end{equation*}
\end{lemma}

\begin{proof} First, observe that for any $u\in \mathcal{D}^{1,2}(\mathds{R}^N_+)$, we have $D_h^i(u)\longrightarrow \partial u/\partial x_i$ in distributional sense. In fact, given $\varphi \in C^\infty_0(\mathds{R}^N_+)$, we obtain
\begin{equation*}
    \int_{\mathds{R}^N_+}D_h^i(u)\varphi\mathrm{d}x=- \int_{\mathds{R}^N_+}uD_{-h}^i(\varphi)\mathrm{d}x\longrightarrow -\int_{\mathds{R}^N_+}u\frac{\partial\varphi}{\partial x_i}\mathrm{d}x= \int_{\mathds{R}^N_+}\frac{\partial u}{\partial x_i}\varphi \mathrm{d}x,\quad \text{as $h\longrightarrow 0$}.
\end{equation*}
Now, we aim to prove that for all $h\in \mathds{R}$ we have
\begin{equation}\label{Dh L2 inequality}
\int_{\mathds{R}^N_+}|D_h^i(u)|^2\mathrm{d}x\leq \int_{\mathds{R}^N_+}\left|\frac{\partial u}{\partial x_i}\right|^2\mathrm{d}x, \quad \forall u\in \mathcal{D}^{1,2}(\mathds{R}^N_+).
\end{equation}
First, consider $u\in C^\infty_\delta(\mathds{R}^N_+)$, fix $x\in \mathds{R}^N_+$ and define $g(t)=u(x+the_i)$, for $t \in [0,1]$. Then, $g^\prime(t)=he_i \cdot\nabla u(x+the_i)$ and
\begin{equation*}
    D_h^i(u)=\frac{1}{h}\int_0^1 g^\prime(t)\mathrm{d}t=\int_0^1 \frac{\partial u}{\partial x_i}(x+the_i)\mathrm{d}t.
\end{equation*}
 Then, by the Jensen inequality (see \cite[Theorem 3.3]{Rudin}) we obtain
 \begin{equation*}
|D_h^i u|^2\mathrm{d}x\leq \int_0^1 \left|\frac{\partial u}{\partial x_i}(x+the_i)\right|^2\mathrm{d}t,
 \end{equation*}
thus, integrating in $\mathds{R}^N_+$, using the Fubini Theorem and a change of variables we obtain 
\eqref{Dh L2 inequality} for $u\in C^\infty_\delta(\mathds{R}^N_+)$. Now, consider $u\in \mathcal{D
}^{1,2}(\mathds{R}^N_+)$ and $(u_n)\subset C^\infty_\delta(\mathds{R}^N_+)$ such that $u_n \longrightarrow u$ in $\mathcal{D}^{1,2}(\mathds{R}^N_+)$. In this case, we have $u_n\longrightarrow u$ a.e. in $\mathds{R}^N_+$ and hence $D_h^i(u_n)\longrightarrow D_h^i(u)$ a.e. in $\mathds{R}^N_+$. From inequality \eqref{Dh L2 inequality} for smooth functions it follows that $(D_h^i(u_n))$ is a Cauchy sequence in $L^2(\mathds{R}^N_+)$ and we must have $D_h^i(u_n)\longrightarrow D_h^i(u)$ in $L^2(\mathds{R}^N_+)$ as $h\longrightarrow 0$. Then \eqref{Dh L2 inequality} holds by density.

Finally, by \eqref{Dh L2 inequality} and the distributional convergence shown earlier, we may assume that $D_h^iu\rightharpoonup \frac{\partial u}{\partial x_i}$ in $L^2(\mathds{R}^N_+)$. Therefore, applying the weak lower semicontinuity of the norm and using again \eqref{Dh L2 inequality}, we obtain
\begin{equation*}
   \left\| \frac{\partial u}{\partial x_i} \right\|_{2,\mathds{R}^N_+}\leq \liminf_{h\longrightarrow 0} \|D_h^i u\|_{2,\mathds{R}^N_+}\leq  \limsup_{h\longrightarrow 0} \|D_h^i u\|_{2,\mathds{R}^N_+}\leq \left\| \frac{\partial u}{\partial x_i} \right\|_{2,\mathds{R}^N_+}.
\end{equation*}
Hence, by the uniform convexity $L^2(\mathds{R}^N_+)$, we conclude the result.
\end{proof}

\begin{lemma}\label{Harnack inequality}
		Assume $(\rho_0)$ with $\rho \in C^1[0,\infty)$ and $\gamma>1$, $a\leq0$, $b>0$, $p\in(1,\infty)$ and $q\in [2,2_\ast]$. Let $u\in E_p(\mathds{R}^N_+) $ be a weak solution of problem \eqref{PG} satisfying $0< u \leq M$ in the half-ball $B_{3\delta}^{+}$ with $1<\delta<2$. Then, there exist constants $C=C(N,M)$ and $\theta_0>3$ such that
		\begin{equation}\label{H inequality}
		\max_{B_1^+}u+\max_{\Gamma_1}u\leq C\left(\| u \|_{\theta_0,B_{2}^+}^{\theta_0}+\| u \|_{{\theta_0},\Gamma_{2}}^{\theta_0} \right)^{1/\theta_0}.
		\end{equation}
		In particular, 
		\begin{equation}\label{decay}
		\lim_{|x^\prime|\longrightarrow +\infty}u(x^\prime)=0\quad \text{for}\quad x^\prime\in\mathds{R}^{N-1}.
		\end{equation}
	\end{lemma}
    
\begin{proof} Let $\beta>2$, $\eta\in C^1(B_{3\delta})$ such that $\mathrm{supp} (\eta)\subset B_\delta^+$ with $0\leq \eta\leq1$. Then, testing with $\varphi=\eta^2u^\beta$ in the weak formulation \eqref{weak solution}, we have
\begin{align*}
    \int_{B_\delta^+}\rho(x_N)\beta \eta^2u^{\beta-1}|\nabla u|^2\mathrm{d}x+2\int_{B_\delta^+}\rho(x_N)\eta u^\beta \nabla u\nabla \eta\mathrm{d}x=&a\int_{B_\delta^+}u^{p-1+\beta}\eta^2\mathrm{d}x\\
    +&b\int_{\Gamma_\delta}u^{q-2}u^{\beta+1}\eta^2\mathrm{d}x^\prime.
\end{align*}
Given that $1\leq \rho \in L^\infty_\mathrm{loc}$, $a\leq 0$ and $u\leq M$, we obtain
\begin{equation*}
    \int_{B_\delta^+}\eta^2u^{\beta-1}|\nabla u|^2\mathrm{d}x\leq \frac{2}{\beta}\int_{B_\delta^+}\eta u^\beta |\nabla u||\nabla \eta|\mathrm{d}x+\frac{bM^{q-2}}{\beta}\int_{\Gamma_\delta}\eta^2 u^{\beta+1}\mathrm{d}x^\prime.
\end{equation*}
Now, consider the Young inequality $cd\leq 2^{-1}( c^2+ d^2)$ for all $c,d>0$. Choosing $c=\eta u^{(\beta-1)/2}|\nabla u|$ and $d=u^{(\beta+1)/2}|\nabla \eta|$, we obtain
\begin{equation*}
\left(1-\frac{1}{\beta}\right)\int_{B_\delta^+}\eta^2u^{\beta-1}|\nabla u|^2\mathrm{d}x\leq  \frac{C}{\beta}\left(\int_{B_\delta^+}u^{\beta+1}|\nabla \eta|^2\mathrm{d}x+\int_{\Gamma_\delta}\eta^2u^{\beta+1}\mathrm{d}x^\prime\right).
\end{equation*}
Observe that $u^{\beta-1}|\nabla u|^2=|2(\beta+1)^{-1}\nabla (u^{\frac{\beta+1}{2}})|^2$ and take $v=u^s$, with $s=(\beta+1)/2$. Thus, since $\beta>2$,
\begin{equation*}
    \frac{1}{s^2}\int_{B_\delta^+}|\eta \nabla v|^2\mathrm{d}x\leq C\left(\int_{B_\delta^+}|v\nabla\eta|^2\mathrm{d}x+\int_{\Gamma_\delta}(\eta v)^2\mathrm{d}x^\prime\right).
\end{equation*}
Summing the term $\int_{\Gamma_\delta}(\eta v)^2\mathrm{d}x$, we obtain
\begin{equation*}
    \left(\|\eta\nabla v\|_{2,B_\delta^+}^2+\|\eta v\|_{2,\Gamma_\delta}^2\right)^{1/2}\leq sC\left(\|v\nabla \eta\|_{2,B_\delta^+}^2+\|\eta v\|_{2,\Gamma_\delta}^2\right)^{1/2}.
\end{equation*}
Let $r_1,r_2$ satisfy the condition $1\leq r_2<\delta\leq r_1\leq 2$. Taking $\eta \equiv 1$ in $B_{r_2}$ and $\eta \equiv 0$ in $\mathds{R}_{+}^{N}\setminus B_{r_1}$, with $| \nabla \eta |\leq 2/(r_1-r_2)$, we achieve
\begin{equation}\label{ineq 2sc}
			\left( \|\nabla v\|_{2,B_{r_2}^+}^2+\|  v \|_{2,\Gamma_{r_2}}^2 \right)^{1/2}\leq \frac{2sC}{r_1-r_2}\left( \| v \|_{2,B_{r_1}^+}^2 + \| v \|_{2,\Gamma_{r_1}}^2\right)^{1/2}.
		\end{equation}
By using Sobolev embedding, Sobolev trace embedding and Friedrichs inequality (see \cite[Lemma 3.1]{CAOM}), we obtain
\begin{equation*}
			\| v \|_{2_\ast,B_{r_2}^+}+\| v \|_{2_\ast,\Gamma_{r_2}}\leq C\left(\| \nabla v \|_{2,B_{r_2}^+}^{2}+ \| v \|_{2,\Gamma_{r_2}}^{2}\right)^{1/2}.
		\end{equation*}
Given that $v=u^s$, applying the inequality above in \eqref{ineq 2sc} we have
\begin{equation}\label{key inequality}
			\left( \int_{B_{r_2}^+} u ^{2_\ast s}\,\mathrm{d}x+\int_{\Gamma_{r_2}}u^{2_\ast s}\,\mathrm{d}x^\prime \right)^{1/2_\ast}\leq \frac{2sC}{r_1-r_2}\left( \int_{B_{r_1}^+}u ^{2s}\,\mathrm{d}x+\int_{\Gamma_{r_1}} u ^{2s}\,\mathrm{d}x^\prime \right)^{1/2}.
		\end{equation}
Now consider the function $\psi: (0,+\infty) \times (0,+\infty)\longrightarrow \mathds{R}$ given by
		\[
		\psi(t,r)=\left( \int_{B_{r}^+}\vert u \vert^{t}\,\mathrm{d}x+\int_{\Gamma_{r}}\vert u\vert^{t}\,\mathrm{d}x^\prime \right)^{1/t}.
		\]
Setting $\theta=2s$ and $\alpha=2_\ast/2$ in \eqref{key inequality}, we have
		\begin{equation}\label{psi inequality}
			\psi(\alpha\theta,r_2)\leq \left(\frac{\theta C}{r_1-r_2}\right)^{2/\theta}\psi(\theta,r_1),\quad \text{for $\theta>3$}.
		\end{equation}
For some $\theta_0>3$ define $\theta_n=\alpha^n\theta_0$ and $r_n=1+2^{-n}$ for $n=1,2,\cdots$. Consequently, from \eqref{psi inequality}, we obtain
\begin{equation}\label{theta induction 1}
			\psi(\theta_{n+1},r_{n+1})\leq \left(\frac{C\alpha^n\theta_0}{r_n-r_{n+1}}\right)^{2/\alpha^n\theta_0}\psi(\theta_{n},r_{n}).
		\end{equation}
By iterating this inequality, we derive the following estimate:     
        
\begin{equation}\label{psi n}
			\begin{alignedat}{2}
				\psi(\theta_{n+1},r_{n+1})&\leq \left(\frac{C\alpha^{n}\theta_0}{r_{n}-r_{n+1}}\right)^{2/(\alpha^{n}\theta_0)}\psi(\theta_n,r_{n})\\
				&\leq \left( C(2\alpha)^{n}\right)^{2\alpha^{-n}/\theta_0}\psi(\theta_n,r_{n})\\
				&=\left(C^\frac{2}{\theta_0}\right)^{\alpha^{-n}}\left( (2\alpha)^\frac{2}{\theta_0} \right)^{n\alpha^{-n}}\psi(\theta_n,r_{n})\\
				&\leq \left(C^\frac{p}{\theta_0}\right)^{\sum_{i=1}^n \alpha^{-i}}\left( (2\alpha)^\frac{2}{\theta_0} \right)^{\sum_{i=1}^n i\alpha^{-i}}\psi(\theta_0,2).
			\end{alignedat}
		\end{equation}
Since $\alpha=2_\ast/2>1$, both series $\sum_{i=1}^\infty\alpha^{-i}$ and $\sum_{i=1}^\infty i\alpha^{-i}$ converge. Therefore, taking the limit as $n\longrightarrow \infty$ in \eqref{psi n}, we obtain
\[
		\max_{B_1^+}u+\max_{\Gamma_1}u=\psi(+\infty,1)\leq C\psi(\theta_0,2)=\left( \int_{B_{2}^+}\vert u \vert^{\theta_0}\,\mathrm{d}x+\int_{\Gamma_{2}}\vert u\vert^{\theta_0}\,\mathrm{d}x^\prime \right)^{1/\theta_0},
		\]
which gives \eqref{H inequality}. For the proof of \eqref{decay}, we observe that by Lemmas \ref{Linfty} and \ref{Linfty 2}, $u \in L^{\theta_0}(\mathds{R}^N_+)\cap L^{\theta_0}(\mathds{R}^{N-1})$. Then, given $\varepsilon>0$, choosing $|x^\prime|$ large enough, by \eqref{H inequality} we obtain $\max_{\Gamma_1(x^\prime)}u<\varepsilon$, which concludes the proof.
\end{proof}	

Now we are ready to prove Theorem \ref{H2_loc regularity theorem}.

\begin{proof}[Proof of Theorem \ref{H2_loc regularity theorem}]: Let $u$ be a nonnegative weak solution of \eqref{PG}. Then, by a change of variables, we obtain
\begin{equation*}
      \int_{\mathds{R}^N_+}\rho(x_N)\nabla D_h^i(u)\nabla \varphi\mathrm{d}x=a\int_{\mathds{R}^N_+}D_h^i(u^{p-1})\varphi\mathrm{d}x+b\int_{\mathds{R}^{N-1}}D_h^i(u^{q-1})\varphi\mathrm{d}x^\prime, \forall \varphi \in C^\infty_\delta(\mathds{R}^N_+)
\end{equation*}
with $i=1,\cdots, N-1$. In each case of the theorem, by density, we are allowed to take $\varphi=D^i_h (u)$, and therefore
\begin{equation}\label{Dhi test}
    \int_{\mathds{R}^N_+}\rho(x_N)|\nabla D_h^i(u)|^2\mathrm{d}x=a\int_{\mathds{R}^N_+}D_h^i(u^{p-1})D_h^i(u)\mathrm{d}x+b\int_{\mathds{R}^{N-1}}D_h^i(u^{q-1})D_h^i(u)\mathrm{d}x^\prime.
\end{equation}

Let us first consider case $(i)$, that is, $\gamma\geq2$, $a>0$, $b\geq$, $p\in [2,2^\ast]$ and $q\in (1,\infty)$. By Lemmas \ref{Linfty} and \ref{Linfty 2}, we know that $u\in L^\infty(\mathds{R}^N_+)$, thus, we have $u^{p-1} \in \mathcal{D}^{1,2}_\rho(\mathds{R}^N_+)$. Then, applying Lemma \ref{limit dq}, we obtain
\[
a\int_{\mathds{R}^N_+}D_h^i(u^{p-1})D_h^i(u)\mathrm{d}x \longrightarrow a(p-1) \int_{\mathds{R}^N_+}u^{p-2}\left(\frac{\partial u}{\partial x_i}\right)^2\mathrm{d}x, \quad  h \longrightarrow 0.
\]
Moreover, since the remaining term on the right-hand side of \eqref{Dhi test} remains nonpositive when $b(q-1)\leq 0$, for $h$ sufficient small we have
\begin{equation*}
    \int_{\mathds{R}^N_+}|\nabla D_h^i(u)|^2 \mathrm{d}x\leq\int_{\mathds{R}^N_+}\rho(x_N)|\nabla D_h^i(u)|^2 \mathrm{d}x\leq a(p-1) \int_{\mathds{R}^N_+}u^{p-2}\left(\frac{\partial u}{\partial x_i}\right)^2\mathrm{d}x<\infty,
\end{equation*}
then, applying \cite[Proposition 9.3]{Brezis} for the half-space, we obtain $\frac{\partial u}{\partial x_i} \in H^1(\mathds{R}^N_+)$ for $i=1,\cdots,N-1$, and thus $\frac{\partial^2u}{\partial x_j\partial x_i} \in L^2(\mathds{R}^N_+)$ for $i=1,\cdots, N-1$ and $j=1,\cdots, N$. Given that the equation
\begin{equation*}
    -\mathrm{div}(\rho(x_N)\nabla u)=au^{p-1}~~\text{in}~~ \mathds{R}^N_+
\end{equation*}
holds in distributional sense, it follows that
\begin{equation*}
   - \frac{\partial}{\partial x_N}\left(\rho(x_N)\frac{\partial u}{\partial x_N}\right)= \sum_{i=1}^{N-1}\rho(x_N)\frac{\partial^2u}{\partial x_i^2}+au^{p-1}.
\end{equation*}
Rearranging the terms, we obtain
\begin{equation*}
    -\rho(x_N)\frac{\partial^2u}{\partial x_N^2}= \sum_{i=1}^{N-1}\rho(x_N)\frac{\partial^2u}{\partial x_i^2}+au^{p-1}+\rho'(x_N)\frac{\partial u}{\partial x_N}.
\end{equation*}
Since all terms on the right-hand side belong to $L^2_{\mathrm{loc}}(\mathds{R}^N_+)$, it follows that $\rho(x_N)\frac{\partial^2u}{\partial x_N^2}$, and therefore $\frac{\partial^2u}{\partial x_N^2}$, belongs to $L^2_{\mathrm{loc}}(\mathds{R}^N_+)$. This completes the proof of case $(i)$. Observe that the same reasoning applies when $\gamma\geq 0$ and $p=2^\ast$.

 %%%%%%%%%%%%%%%%%%%%%%%%%%%%%%

Now, turning to item $(ii)$, since $a\leq 0$, by \eqref{Dhi test} we have
\begin{equation*}
\int_{\mathds{R}^N_+}\rho(x_N)|\nabla D_h^i(u)|^2\mathrm{d}x\leq b\int_{\mathds{R}^{N-1}}D_h^i(u^{q-1})D_h^i(u)\mathrm{d}x^\prime.
\end{equation*}
Consider the following inequality, valid for $s,t>0$:
\begin{equation*}
    |s^{q-1}-t^{q-1}|\leq (q-1)(\lambda s+(1-\lambda)t)^{q-2}|s-t|,
\end{equation*}
for some $\lambda\in (0,1).$ Thus, we deduce
\begin{align}\label{key Nabla D}
\nonumber    \int_{\mathds{R}^N_+}\rho(x_N)|\nabla D_h^i(u)|^2\mathrm{d}x\leq& b\int_{\mathds{R}^{N-1}}\frac{u^{q-1}(x+he_i)-u^{q-1}(x)}{h}D_h^i(u)\mathrm{d}x^\prime\\ 
    \leq& b(q-1)\int_{\mathds{R}^{N-1}}(\lambda u(x+he_i)+(1-\lambda)u)^{q-2}|D_h^i(u)|^2\mathrm{d}x^\prime.
\end{align}
By Lemmas \ref{Linfty} and \ref{Linfty 2}, we know that $u\in L^\infty(\mathds{R}^N_+)\cap L^\infty(\mathds{R}^{N-1})$ and due to Lemma \ref{Harnack inequality}, we can choose $R>0$ such that
\begin{equation}\label{L infty ineq}
    \|u\|_{\infty, \mathds{R}^{N-1\backslash}\Gamma_R}^{q-2}<\frac{\gamma-1}{b2^q(q-1)}.
\end{equation}
Furthermore,
\begin{align*}
    \int_{\mathds{R}^{N-1}}(\lambda u(x+he_i)+(1-\lambda)u(x))^{q-2}|D_h^i(u)|^2\,\mathrm{d}x^\prime
    \leq\ & 
       2^{q-2} \|u\|^{q-2}_{\infty,\Gamma_R} \int_{\Gamma_R} |D_h^i(u)|^2\,\mathrm{d}x^\prime\\
        +&2^{q-2} \|u\|^{q-2}_{\infty, \mathds{R}^{N-1} \setminus \Gamma_R}\int_{\mathds{R}^{N-1}} |D_h^i(u)|^2\,\mathrm{d}x^\prime.
\end{align*}
Now, observe that due to inequality \eqref{Hardy} and \eqref{L infty ineq} we have
\begin{align*}
    \int_{\mathds{R}^{N-1}}(\lambda u(x+he_i)+(1-\lambda)u(x))^{q-2}|D_h^i(u)|^2\,\mathrm{d}x^\prime
    \leq & 2^{q-2}\|u\|^{q-2}_{\infty,\Gamma_R} \int_{\Gamma_R} |D_h^i(u)|^2\,\mathrm{d}x^\prime\\
        +& 2^{q-2}\|u\|^{q-2}_{\infty,\mathds{R}^{N-1} \setminus \Gamma_R} \left(\frac{2}{\gamma-1}\right)\int_{\mathds{R}^{N}_+}\rho(x_N) |\nabla D_h^i(u)|^2\,\mathrm{d}x\\
        \leq &  
        2^{q-2}\|u\|^{q-2}_{\infty,\Gamma_R} \int_{\Gamma_R} |D_h^i(u)|^2\,\mathrm{d}x^\prime\\
        +& \frac{1}{2b(q-1)}\int_{\mathds{R}^{N}_+}\rho(x_N) |\nabla D_h^i(u)|^2\,\mathrm{d}x.
\end{align*}
Applying this in \eqref{key Nabla D}, we obtain
\begin{equation*}
    \int_{\mathds{R}^N_+}\rho(x_N)|\nabla D_h^i(u)|^2\mathrm{d}x \leq 2^{q-1}\|u\|^{q-2}_{\infty,\mathds{R}^{N-1}}\int_{\Gamma_R}|D_h^i(u)|^2\mathrm{d}x^\prime.
\end{equation*}
By Lemma \ref{Lib}, $u \in C^{1,\alpha}_{\mathrm{loc}}(\overline{\mathds{R}^N_+})$ and we obtain 
\begin{equation*}
\int_{\mathds{R}^N_+}|\nabla D_h^i(u)|^2\mathrm{d}x\leq \int_{\mathds{R}^N_+}\rho(x_N)|\nabla D_h^i(u)|^2\mathrm{d}x\leq C.
\end{equation*}
Arguing as in the first case, we conclude that $\frac{\partial u}{\partial x_i} \in H^1(\mathds{R}^N_+)$ for $i=1,\cdots N-1$ and $\frac{\partial^2u}{\partial x_N^2}\in L^2_{\mathrm{loc}}(\overline{\mathds{R}^N_+})$, which concludes the result.
\end{proof}

%%%%%%%%%%%%%%%%%%%%%%%%%%%%%%%%%%%%%%%%%%%%%%%%%%%%%%%%%%%%%%%%%%%%%%%5

\section{Proof of existence results}

\subsection{Proof of Theorems \ref{Problem1-Existence}} We begin by noting that, using standard arguments and Lemma \ref{embedding}, the functional $I$ has the Mountain Pass geometry.

 \begin{lemma}\label{mp geometry} Assume $(\rho_0)$ with $\gamma>1$, $a\leq0$ and $b>0$. If $1<p$ and $\max\{2,p\}<q<2_\ast$, the functional $I:E_p(\mathds{R}^N_+)\longrightarrow \mathds{R}$ has the mountain pass geometry, that is, there exists a constants $C_0>0$ and $r_0>0$ sufficiently small such that 
 \begin{enumerate}
     \item [$(i)$] $I(u)\geq C_0>$ if $\|u\|_{E_p(\mathds{R}^N_+)}=r_0$;
      \item [$(ii)$] For each $u\not\equiv0$ it holds $\displaystyle\lim_{t\longrightarrow+\infty} I(tu)=-\infty$.
 \end{enumerate}
   \end{lemma}

 \begin{proof} By Lemma \ref{embedding}, we obtain

\begin{align*}
    I(u)= \frac{1}{2}\|u\|^2 -\frac{b}{q}\int_{\mathds{R}^{N-1}}|u|^q\mathrm{d}x^\prime+\frac{|a|}{p}\int_{\mathds{R}^N_+}|u|^p\mathrm{d}x
    \geq \frac{1}{2}\|u\|^2 -C\|u\| ^q.
\end{align*}
 Since $2<q$, we have $I(u)\geq C_0>0$ for $\|u\|_{E_p(\mathds{R}^N_+)} $ sufficiently small. 
 
 Now, let $u \in E_p(\mathds{R}^N_+)\backslash\{0\}$ and $t>0$. Suppose first that $p\leq 2$. Since $q>\max\{2,p\}$, we see that 
 \begin{equation*}
     I(tu)=t^2\left(\frac{1}{2}\|u\|^2 -\frac{bt^{q-2}}{q}\int_{\mathds{R}^{N-1}}|u|^q\mathrm{d}x^\prime+ \frac{|a| t^{p-2}}{p} \int_{\mathds{R}^N_+} |u|^p  \mathrm{d}x \right)\longrightarrow -\infty\quad\mbox{as}\quad t\longrightarrow \infty.
 \end{equation*}
Similarly, we obtain the same for $p<2$.
\end{proof}

\begin{remark}\label{Nonnegative ps sequence}
By Lemma \ref{mp geometry}, the minimax level is well defined:
\begin{equation*}
   0< c= \inf_{g \in \Lambda } \max_{t \in [0,1]}I(g(t)),
\end{equation*}
where 
\begin{equation*}
\Lambda=\{g \in C([0,1], E_p(\mathds{R}^N_+)): g(0)=0~~\text{and}~~ I(g(1))<0\}.
\end{equation*} 
Now, observe that since $I(u)=I(|u|)$ for all $u \in \mathcal{D}^{1,2}_\rho(\mathds{R}^N_+)$, it follows that 
\begin{equation*}
    c = \inf_{g\in\Lambda^+}\max_{t\in[0,1]}I(g(t)),
\end{equation*}
where $\Lambda^+=\{g \in C([0,1], E_p^+(\mathds{R}^N_+)): g(0)=0~~\text{and}~~ I(g(1))<0\}$ and $E_p^+(\mathds{R}^N_+):= \{u \in E_p(\mathds{R}^N_+): u\geq 0\}$. In this case, we can obtain a Palais–Smale sequence $(u_n)$ with $u_n\geq 0$. Indeed, given $\varepsilon>0$, by applying the Ekeland variational principle to the functional $\Phi:\Lambda^+\longrightarrow \mathds{R}$, defined by $\Phi(g)=\max_{t \in [0,1]}I(g(t))$, we obtain $h \in \Lambda^+$ such that
\begin{equation*}
    \Phi(g)>\Phi(h)-\varepsilon^{1/2}d(h,g),~~\forall g\in \Lambda^+\setminus\{h\}, 
\end{equation*}
where $d(\cdot,\cdot)$ is a metric on the space $\Lambda$. Given that $\Phi(g)=\Phi(|g|)$ and $d(h,g)\geq d(h,|g|)$ for all $g\in \Lambda$, we also have
\begin{equation*}
    \Phi(g)>\Phi(h)-\varepsilon^{1/2}d(h,g),~~\forall g\in \Lambda\setminus\{h\}.
\end{equation*}
Now, proceeding as in \cite[Theorem 4.3]{mawhin-willem}, we find $v\geq 0$ such that
\begin{equation*}
    c-\varepsilon<I(v)<c+\varepsilon\quad\text{and}\quad \|I'(v)\|\leq \varepsilon^{1/2}.
\end{equation*}
Therefore, we can construct a sequence $(u_n)\subset E_p(\mathds{R}^N_+)$ satisfying
\begin{equation}\label{nonnegative Ps seq}
u_n\geq 0, \quad    I(u_n)\longrightarrow c\quad\text{and}\quad I^\prime(u_n)\longrightarrow 0.    
\end{equation}
\end{remark}

\begin{lemma}\label{limitacao inferior} The sequence $(u_n)$ is bounded in $E_p(\mathds{R}^N_+)$ and there exists $C_0>0$ such that
\begin{equation}\label{desigualdade rho}
0<C_0\leq \int_{\mathds{R}^{N-1}}|u_n|^q\mathrm{d}x^\prime,  
\end{equation}
for $n$ sufficiently large.
\end{lemma}

\begin{proof} Let $\max\{2,p\}<\mu<q$. For $n$ large enough, we have 
\begin{align*}
c+1+\|u_n\|_{E_p(\mathds{R}^N_+)} \geq& I(u_n)-\frac{1}{\mu}I^\prime(u_n)u_n\\
=&\left(\frac{1}{2}-\frac{1}{\mu}\right)\|u_n\|^2+b\left(\frac{1}{\mu}-\frac{1}{q}\right)\int_{\mathds{R}^{N-1}}|u_n|^q\mathrm{d}x^\prime+|a|\left(\frac{1}{p}-\frac{1}{\mu}\right)\int_{\mathds{R}^N_+}|u_n|^p\mathrm{d}x,
\end{align*}
then 
\begin{equation*}
    c+1\geq c_1\|u_n\|^2+c_2\|u\|_{p,\mathds{R}^N_+}^p-\|u\|_{E_p(\mathds{R}^N_+)},
\end{equation*}
for some constants $c_1,c_2>0$, which implies
that $\|u_n\|_{E_p(\mathds{R}^N_+)}$ is bounded. On the other hand, let $0<\lambda<\min\{2,p\}$. Then, for $n$ sufficiently large, it holds that 
$$
0<\frac{c}{2}\leq I(u_n)-\frac{1}{\lambda} I'(u_n)u_n\leq b\left(\frac{1}{2}-\frac{1}{q}\right)\int_{\mathds{R}^{N-1}}|u_n|^q\mathrm{d}x^\prime,
$$
and this concludes the proof.
\end{proof}

\begin{lemma}\label{key lemma}  Given $y\in \mathds{R}^{N-1}$, consider
\begin{equation*}\label{BGamma norm}
    \|u\|_{B^+_1, \Gamma, y}^2:=\|\rho(x_N)^{1/2}\nabla u\|_{2, B_1^+(y)}^2+\|u\|_{2,  \Gamma_1(y)}^2.
\end{equation*}
Then
\begin{equation}\label{BGamma inequality}
    \|u\|_{r, \Gamma_1(y)}\leq C  \|u\|_{B^+_1, \Gamma, y},
\end{equation}
for all $r \in [2,2_\ast]$ and $u  \in E_p(\mathds{R}^N_+)$.
\end{lemma}

\begin{proof} By the embedding $H^1(B_1^+(y))\hookrightarrow L^q(\partial B_1^+(y))$, we obtain
\begin{equation*}
    \|u\|_{q, \Gamma_1(y)}\leq \|u\|_{q, \partial B_1^+(y)}\leq C(\|\rho(x_N)^{1/2}\nabla u\|_{2, B_1^+(y)}^2+\|u\|_{2,B_1^+(y)}^2)^{1/2}.
\end{equation*}
Furthermore, by the Friedrichs inequality(see \cite[Lemma 3.1]{CAOM}), it follows that:
\[
\|u\|_{2,B_1^+(y)}\leq C(\|\nabla u\|_{2, B_1^+(y)}^2+\|u\|_{2, \Gamma_1(y)}^2)^{1/2}\leq C\|u\|_{B_1^+,\Gamma,y}.
\]
Then, combining the two inequalities above we obtain \eqref{BGamma inequality}.
\end{proof}

\begin{lemma}\label{sup}Let $(u_n)$ the sequence in \eqref{nonnegative Ps seq}, then there exists $C>0$ such that 
\begin{equation}\label{sup inequality}
    0<C \leq \sup_{y \in \mathds{R}^{N-1}}\int_{ \Gamma_1(y)}u_n^2\mathrm{d}x^\prime.
\end{equation}
\end{lemma}

\begin{proof}Let $r\in (q,2_\ast)$. By interpolation inequality, we have
\begin{align*}
    \|u\|_{q, \Gamma_1(y)}^q \leq \|u\|_{2, \Gamma_1(y)}^{(1-\alpha)q}\|u\|_{r, \Gamma_1(y)}^{\alpha q},
\end{align*}
with $\alpha= \frac{r}{q}\frac{q-2}{r-2}$. First, consider the case where $r_\ast:=4(r-1)/r\leq q$. This condition is equivalent to $\alpha q\geq 2$. By Lemma \ref{key lemma}, we obtain
\begin{align*}
     \|u\|_{q, \Gamma_1(y)}^q &\leq C \|u\|_{2, \Gamma_1(y)}^{(1-\alpha)q}\|u\|_{B_1^+,\Gamma,y}^{\alpha q}\\
     &\leq C \left(\sup_{y \in \mathds{R}^{N-1}}\int_{ \Gamma_1(y)}u^2\mathrm{d}x^\prime\right)^{(1-\alpha)q/2}\|u\|_{B_1^+,\Gamma,y}^{\alpha q-2}\|u\|_{B_1^+,\Gamma,y}^2.
\end{align*}
By \eqref{Hardy} and $(\rho_0)$ we obtain
\begin{align*}
    \|u\|_{B_1^+,\Gamma, y}^2&= \int_{B_1^+(y)}\rho(x_N)|\nabla u|^2\mathrm{d}x+\int_{ \Gamma_1(y)}u^2\mathrm{d}x^\prime\\
    &\leq \int_{\mathds{R}^N_+}\rho(x_N)|\nabla u|^2\mathrm{d}x+\int_{\mathds{R}^{N-1}}u^2\mathrm{d}x^\prime\\
    &\leq C \|u\|  ^2.
\end{align*}
Thus, since $\alpha q\geq 2$
\begin{equation}\label{key inequality 2}
    \|u\|_{q, \Gamma_1(y)}^q \leq C \left(\sup_{y \in \mathds{R}^{N-1}}\int_{ \Gamma_1(y)}u^2\mathrm{d}x^\prime\right)^{(1-\alpha)q/2} \|u\|  ^{\alpha q-2}\|u\|_{B_1^+,\Gamma,y}^2.
\end{equation}
Now, consider a family $\{B_1^+(y)\}_{y \in \mathds{R}^{N-1}}$ such that the collection $\{ \Gamma_1(y)\}_{y \in \mathds{R}^{N-1}}$ covers $\mathds{R}^{N-1}$ and ensures that each point in $\mathds{R}^{N-1}$ is contained in at most $N$ of these balls. By summing the inequalities in \eqref{key inequality 2} over this family, we find
\begin{equation}\label{key inequality 3}
     \int_{\mathds{R}^{N-1}}|u|^q\mathrm{d}x^\prime \leq CN \left(\sup_{y \in \mathds{R}^{N-1}}\int_{ \Gamma_1(y)}u^2\mathrm{d}x^\prime\right)^{(1-\alpha)q/2} \|u\|  ^{\alpha q}.
\end{equation}
Applying this inequality with $u=u_n$ and using the fact that $(u_n)$ is bounded, we have
\begin{equation*}
     \int_{\mathds{R}^{N-1}}|u_n|^q\mathrm{d}x^\prime \leq C \left(\sup_{y \in \mathds{R}^{N-1}}\int_{ \Gamma_1(y)}u_n^2\mathrm{d}x^\prime\right)^{(1-\alpha)q/2}.
\end{equation*}
Then \eqref{sup inequality} follows by Lemma \ref{limitacao inferior}.

Now consider the case $r_\ast=4(r-1)/r > q$. In this case, we have $q<r_\ast<r$ and
\begin{equation*}
    \|u\|_{q,\mathds{R}^{N-1}}^q\leq \|u\|_{2,\mathds{R}^{N-1}}^{(1-\alpha)q}\|u\|_{r_\ast,\mathds{R}^{N-1}}^{\alpha q},
\end{equation*}
with $\alpha=\frac{r}{r_\ast}\frac{r_\ast-2}{r-2}$. Since $\alpha r_\ast=2$ , taking $q=r_\ast$ in \eqref{key inequality 3}, we have 
\begin{equation*}
      \int_{\mathds{R}^{N-1}}|u|^{r_\ast}\mathrm{d}x^\prime \leq CN \left(\sup_{y \in \mathds{R}^{N-1}}\int_{ \Gamma_1(y)}u^2\mathrm{d}x^\prime\right)^{(1-\alpha)r_\ast/2} \|u\|  ^{2}.
\end{equation*}
Thus,
\begin{equation*}
    \int_{\mathds{R}^{N-1}}|u|^q\mathrm{d}x^\prime\leq CN^{\theta q/r_\ast}\|u\|_{2,\mathds{R}^{N-1}}^{(1-\alpha)q} \left[\left(\sup_{y \in \mathds{R}^{N-1}}\int_{ \Gamma_1(y)}u^2\mathrm{d}x^\prime\right)^{(1-\alpha)r_\ast/2} \|u\| ^2\right]^{\theta q/r_\ast}.
\end{equation*}
Therefore, taking $u=u_n$, using Lemma \ref{embedding}, \eqref{desigualdade rho} and the fact that $(u_n)$ is bounded, we obtain \eqref{sup inequality}.
\end{proof}

We are now ready to prove Theorem~\ref{Problem1-Existence}

\begin{proof}[Proof of Theorem \ref{Problem1-Existence}] By Lemma \ref{sup}, there exists a sequence $(y_n)\subset \mathds{R}^{N-1}$ such that
\begin{equation*}
0<C\leq\int_{\Gamma_1(y_n)}u_n^2\mathrm{d}x^\prime.
\end{equation*}
Defining $v_n(x^\prime,x_N)=u_n(x^\prime+y_n,x_N)$, we see that 
\begin{equation}\label{v inequality}
\int_{\Gamma(0)}v_n^2\mathrm{d}x^\prime=\int_{\Gamma_1(y_n)}u_n^2 \mathrm{d}x^\prime\geq C>0.
\end{equation}
Since $(u_n)$ is such that $I(u_n)\longrightarrow c$ and $I'(u_n)\longrightarrow 0$, taking into account that  
\[
\int_{\mathds{R}^N_+}\rho(x_N)|\nabla v_n|^2\mathrm{d}x=\int_{\mathds{R}^N_+}\rho(x_N)|\nabla u_n|^2\mathrm{d}x,\quad \int_{\mathds{R}^N_+}|v_n|^p\mathrm{d}x= \int_{\mathds{R}^N_+}|u_n|^p\mathrm{d}x,
\]
and
\[
\int_{\mathds{R}^{N-1}}|v_n|^q\mathrm{d}x^\prime= \int_{\mathds{R}^{N-1}}|u_n|^q\mathrm{d}x^\prime,
\]
we see that  $I(v_n)=I(u_n)\longrightarrow c$. Furthermore,
\[
\sup_{\|\varphi\|_{E_p(\mathds{R}^N_+)} =1 }|I'(v_n)\varphi|=\sup_{\|\varphi\|_{E_p(\mathds{R}^N_+)} =1}|I'(u_n)\varphi(x-y_n)|\leq \sup_{\|\varphi\|_{E_p(\mathds{R}^N_+)} =1}\|I'(u_n)\|\|\varphi\|\longrightarrow 0.
\]
Hence, $(v_n)$ also is a Palais–Smale sequence at level $c$ and so $(v_n)$ is bounded. Passing if necessary to a subsequence, we can assume that $v_n\rightharpoonup v$ in $E_p(\mathds{R}^N_+)$ with $v\geq 0$, due to \eqref{nonnegative Ps seq}. Given that, $I'(v_n)\longrightarrow 0$, we can pass to the limit and obtain $I'(v)\varphi =0$ for all $\varphi \in C^\infty_\delta(\mathds{R}^N_+)$. Now, we observe that $v_n\longrightarrow v$ in $L^2_\mathrm{loc}(\mathds{R}^{N-1})$, then by \eqref{v inequality} we have
\[
\int_{\Gamma(0)}v^2\mathrm{d}x^\prime\geq C>0.
\]
Thus, $v$ is a nontrivial nonnegative weak solution of \eqref{Coro1}. By Theorem \ref{Holder regularity result}, $u\in C^{2,\alpha}_{\mathrm{loc}}(\overline{\mathds{R}^N_+})$ and $u > 0$ in $\overline{\mathds{R}^N_+}$, which completes the proof of Theorem.
\end{proof}

\subsection{Proof of Theorem~\ref{Problem2-Existence} and \ref{Existence 3}}

In this section, we seek solutions for \eqref{PG} in the space $\mathcal{R}^{1,2}_\rho(\mathds{R}^N_+)$.

\begin{proof}[Proof of Theorem \ref{Problem2-Existence}:] 
 For each case of Theorem, we can argue as in Lemmas \ref{mp geometry} and \ref{limitacao inferior}, and conclude that the functional $I$ satisfies the mountain pass geometry and that every Palais–Smale sequence is bounded. Following the reasoning in Remark \ref{nonnegative Ps seq}, there exists a sequence $(u_n)$ such that
 \begin{equation*}
     u_n\geq0,\quad I(u_n)\longrightarrow c>0,\quad\text{and}\quad I^\prime(u_n)\longrightarrow 0,
 \end{equation*}
 where $c=\inf_{g \in \Lambda } \max_{t \in [0,1]}I(g(t))$ and $\Lambda=\{g \in C([0,1], \mathcal{R}^{1,2}_{\rho}(\mathds{R}^N_+)): g(0)=0~~\text{and}~~I(g(1))<0\}$. Given that $(u_n)$ is bounded, there exists $u\in \mathcal{R}^{1,2}_\rho(\mathds{R}^N_+)$ such that, up to a subsequence, $u_n\rightharpoonup u\geq 0$. Thus, $I^\prime(u_n)\varphi\longrightarrow I^\prime(u)\varphi$, for all $\varphi\in C^\infty_\delta(\mathds{R}^N_+)$, and $u$ is a weak solution. By Lemma \ref{Compact embedding} and since $\langle I^\prime(u_n)-I^\prime(u), u_n-u \rangle=o(1)$, we obtain $u_n\longrightarrow u$ in $\mathcal{D}^{1,2}_\rho(\mathds{R}^N_+)$. Therefore, $I(u_n)\longrightarrow I(u)=c>0$ and $u\neq 0$. Finally, $u\in C^{2,\alpha}_{\mathrm{loc}}(\mathds{R}^N_+)$ and is positive in $\mathds{R}^N_+$, as ensured by Theorem \ref{Holder regularity result}. 
\end{proof}

\begin{proof}[Proof of Theorem \ref{Existence 3}:] Considering the space $\mathcal{R}^{1,2}_\rho(\mathds{R}^N_+)\cap E_q(\mathds{R}^{N-1})$ with the norm $\|\cdot\|+|b|\|\cdot\|_{q,\mathds{R}^{N-1}}$, we can argue as in the previous theorems to prove that $I$ satisfies the mountain pass geometry and that there exists a nonnegative Palais-Smale sequence $(u_n)\subset \mathcal{R}^{1,2}_\rho(\mathds{R}^N_+)\cap E_q(\mathds{R}^{N-1})$ at level $c>0$ and such that $\|u_n\|+|b|\|u_n\|_{q,\mathds{R}^{N-1}}\leq C$. Then, there exists a subsequence such that $u_n \rightharpoonup u\geq 0$ in $\mathcal{R}^{1,2}_\rho(\mathds{R}^N_+)\cap E_q(\mathds{R}^{N-1})$ and in $L^q(\mathds{R}^{N-1})$. Thus $I^\prime(u_n)\varphi\longrightarrow I^\prime(u)\varphi$, for all $\varphi\in C^\infty_\delta(\mathds{R}^N_+)$ and then $u$ is a critical point for $I$. Using a density argument, we test the weak formulation with $u$, obtaining
\begin{equation*}
    \int_{\mathds{R}^N_+}\rho(x_N)|\nabla u|^2\mathrm{d}x+|b|\int_{\mathds{R}^{N-1}}|u|^q\mathrm{d}x^\prime=a\int_{\mathds{R}^N_+}|u|^p\mathrm{d}x.
\end{equation*}
Given that $I^\prime(u_n)u_n=o(1)$, using the compactness in $L^p(\mathds{R}^N_+)$ provided by Lemma~\ref{Compact embedding}, we obtain
\begin{align*}
\int_{\mathds{R}^N_+}\rho(x_N)|\nabla u_n|^2\mathrm{d}x+|b|\int_{\mathds{R}^{N-1}}|u_n|^q\mathrm{d}x^\prime=&a\int_{\mathds{R}^N_+}|u_n|^p\mathrm{d}x+o(1)\\
    =& \int_{\mathds{R}^N_+}\rho(x_N)|\nabla u|^2\mathrm{d}x+|b|\int_{\mathds{R}^{N-1}}|u|^q\mathrm{d}x^\prime+o(1).
\end{align*}
Let $l=\lim_{n\rightarrow \infty} \|u_n\|^2$. By weak convergence, we have $\|u\|^2\leq \liminf_{n\rightarrow \infty}\|u_n\|^2=l$. On the other hand,
\begin{align*}
    l+\|u\|_{q,\mathds{R}^{N-1}}^q\leq& \liminf_{n\rightarrow \infty} \|u_n\|^2+\liminf_{n\rightarrow \infty}\|u_n\|_{q,\mathds{R}^{N-1}}^q\\
    \leq& \liminf_{n\rightarrow \infty}(\|u_n\|^2+\|u_n\|_{q,\mathds{R}^{N-1}}^q)\\
    =& \|u\|^2+\|u\|_{q,\mathds{R}^{N-1}}^q.
\end{align*}
Thus, $\lim_{n\rightarrow \infty} \|u_n\|^2=\|u\|^2$ and hence $\lim_{n\rightarrow \infty}\|u_n\|_{q,\mathds{R}^{N-1}}^q=\|u\|_{q,\mathds{R}^{N-1}}^q$. Therefore, $I(u_n)\longrightarrow I(u)=c>0$, which implies that $u\neq 0$. By Theorem \ref{Holder regularity result}, we have $u\in C^{2,\alpha}_{\mathrm{loc}}(\overline{\mathds{R}^N_+})$ and $u$ is positive in $\overline{\mathds{R^N_+}}$.
\end{proof}

%%%%%%%%%%%%%%%%%%%%%%%%%%%%%%%%%%%%%%%%%%%%%%%%%%%%%%%%%%%%%%%%%%%%%%%%%%%%%%%%%%
%
%                  nonexistence RESULTS
%
%%%%%%%%%%%%%%%%%%%%%%%%%%%%%%%%%%%%%%%%%%%%%%%%%%%%%%%%%%%%%%%%%%%%%%%%%%%%%%%%%%

\section{Pohozaev identity and nonexistence results}    
    
\begin{proof}[Proof of Pohozaev identity]: Let $R>0$ and $\Omega_R \subset \overline{\mathds{R}^N_+}$ be such that $B_R^+\subset\Omega_R$, $\Gamma_R\subset \partial \Omega_R$ and $\partial \Omega$ is sufficiently smooth. First, consider $u$ a sufficiently smooth function. Then, by a standard calculation, we have in $\Omega_R$:
		\begin{align*}
			{
				\rm{div}}\left((x\cdot\nabla u)\rho(x_N)\nabla u\right)&=(x\cdot\nabla u)\mathrm{div}(\rho(x_N)\nabla u)+\nabla(x\cdot\nabla u)\cdot\rho(x_N)\nabla u\\
			&=(x\cdot\nabla u)\mathrm{div}(\rho(x_N)\nabla u)+\rho(x_N)|\nabla u|^2+\frac{\rho(x_N)}{2}\left(x\cdot\nabla (|\nabla u|^2)\right)
		\end{align*}
		and
		\begin{align*}
		{\rm{div}}\left(x\rho(x_N)|\nabla u|^2\right)&=\mathrm{div}(x\rho(x_N))|\nabla u|^2+\rho(x_N)x\cdot\nabla (|\nabla u|^2\\
        &=\left(N\rho(x_N)+\rho'(x_N)x_N\right)|\nabla u|^2+ \rho(x_N)x\cdot\nabla (|\nabla u|^2).
		\end{align*}
Now, let $\psi_R \in C_0^\infty(B_R)$ be restricted to $\mathds{R}^N_+$ such that  
        \begin{align*}
			\psi_{R}(x)=\left\{
			\begin{aligned}
				1&\quad&\text{if}&\quad |x|\leq R/2,\\
				0&\quad&\text{if}&\quad |x|\geq R,
			\end{aligned}
			\right.
		\end{align*}
with $|\nabla \psi_{R}(x)|\leq2/R$, and define the vector field $V_{R}: \Omega_R \longrightarrow\mathds{R}^N$ by
		\begin{equation*}%\label{vector}
		    V_{R}(x):=\left[(x\cdot\nabla u)\rho(x_N)\nabla u - \frac{x\rho(x_N)}{2}|\nabla u|^2\right]\psi_{R}.
		\end{equation*}
From the previous computations, it follows that		\begin{align*}\label{div V}
		    -\mathrm{div}(V_R)=&-\left[(x\cdot\nabla u)\mathrm{div}(\rho(x_N)\nabla u)+\frac{2-N}{2}\rho(x_N)|\nabla u|^2-\frac{\rho'(x_N)x_N}{2}|\nabla u|^2\right]\psi_R\\
            -&\left[(x\cdot\nabla u)\rho(x_N)\nabla u - \frac{x\rho(x_N)}{2}|\nabla u|^2\right]\nabla\psi_R
		\end{align*}
in $\Omega_{R}$. Thus,
\begin{align*}
\int_{\Gamma_R}V_{R}\cdot\nu\,\mathrm{d}x^\prime&=\int_{\Omega_R}\mathrm{div}(V_R)\mathrm{d}x\\
            &=\int_{ B_R^+}\left[(x\cdot\nabla u)\mathrm{div}(\rho(x_N)\nabla u)+\frac{2-N}{2}\rho(x_N)|\nabla u|^2-\frac{\rho'(x_N)x_N}{2}|\nabla u|^2\right]\psi_R\,\mathrm{d}x\\ 
			&+\int_{B_{R}^+\setminus B_{R/2}}\left[(x\cdot\nabla u)\rho(x_N)\nabla u - \frac{x\rho(x_N)}{2}|\nabla u|^2\right]\nabla\psi_{R}\,\mathrm{d}x.
\end{align*}

Now let $u \in H^2_{\mathrm{loc}}(\overline{\mathds{R}^N_+})$ and $(u_n) \subset C^\infty(\overline{\Omega_R})$ such that $u_n \to u$ in $H^2(\overline{\Omega_R})$. Applying the above equation to $u_n$ and passing the limit yield the same for $u \in H^2_{\mathrm{loc}}(\overline{\mathds{R}^N_+})$. Indeed, 
\begin{align*}
    \left|\int_{\Gamma_R}V_R^n\cdot\nu\mathrm{d}x^\prime-\int_{\Gamma_R}V_R\cdot\nu\mathrm{d}x^\prime\right|\leq&\int_{\Gamma_R} |((x\cdot\nabla u_n)\nabla u_n-(x\cdot\nabla u)\nabla u)\cdot \nu|\mathrm{d}x^\prime\\
    =&\int_{\Gamma_R} \left|(x\cdot\nabla u)\frac{\partial u}{\partial x_N}-(x\cdot \nabla u_n)\frac{\partial u_n}{\partial x_N}\right|\mathrm{d}x^\prime\\
    \leq& C\|\nabla u\|_{2,\Gamma_R}\left\|\frac{\partial u}{\partial x_N}-\frac{\partial u_n}{\partial x_N}\right\|_{2,\Gamma_R}+\left\| \frac{\partial u_n}{\partial x_N}\right\|_{2,\Gamma_R}\|x(\nabla u_n-\nabla u)\|_{2,\Gamma_R}\\
    \leq& C_1\left\|\frac{\partial u}{\partial x_N}-\frac{\partial u_n}{\partial x_N}\right\|_{H^1(\Omega_R)}+C_2\|\nabla u_n-\nabla u\|_{H^1(\Omega_R)}\\
    \longrightarrow& 0,
\end{align*}
\begin{align*}
    \int_{B_R^+}(x\cdot\nabla u_n)\mathrm{div}(\rho(x_N)\nabla u_n)\psi_R\mathrm{d}x=& \int_{B_R^+}(x\cdot\nabla u_n)\left(\rho(x_N)\Delta u_n-\rho^\prime(x_N)\frac{\partial u_n}{\partial x_N}\right)\psi_R\mathrm{d}x\\
    \longrightarrow& \int_{B_R^+}(x\cdot\nabla u)\left(\rho(x_N)\Delta u-\rho^\prime(x_N)\frac{\partial u}{\partial x_N}\right)\psi_R\mathrm{d}x\\
    =& \int_{B_R^+}(x\cdot\nabla u)\mathrm{div}(\rho(x_N)\nabla u)\psi_R\mathrm{d}x,
\end{align*}
and clearly the other terms also converges for the desired limit.

Since $|\nabla \psi_{R}(x)|\leq2/R$ and $\rho(x_N)|\nabla u|^2\in L^1(\mathds{R}^N_+)$, we have   
		\begin{align*}
			\left|\int_{B_{R}^+\setminus B_{R/2}}\left[(x\cdot\nabla u)\rho(x_N)\nabla u - \frac{1}{2}x\rho(x_N)|\nabla u|^2\right]\nabla\psi_{R}\,\mathrm{d}x\right|          
			&\leq 3\int_{B_{R}^+\setminus B_{R/2}}\rho(x_N)|\nabla u|^2\,\mathrm{d}x\longrightarrow 0,
		\end{align*}
		as $R\longrightarrow+\infty$.
Moreover, since
\begin{equation*}
    \int_{B_R^+}\rho(x_N)\nabla u\nabla \varphi\mathrm{d}x=\int_{B_R^+}f(u)\varphi\mathrm{d}x,\quad \forall \varphi \in C^{\infty}_0(B_R^+),
\end{equation*}
and $u \in H^2(B_R^+)$, integration by parts yields $-\mathrm{div}(\rho(x_N)\nabla u)=f(u)$ a.e. in $B_R^+$. Hence, we conclude        
\begin{equation*}
\int_{\Gamma_R}V_{R}\cdot\nu\,\mathrm{d}x^\prime =-\int_{ B_R^+}\left[f(u)(x\cdot\nabla u)+\frac{N-2}{2}\rho(x_N)|\nabla u|^2+\frac{\rho'(x_N)x_N}{2} |\nabla u|^2\right]\psi_{R}\,\mathrm{d}x+ o_R(1).    
\end{equation*}        
        
        Let $(R_k)_{k\in\mathds{N}}$  be a sequence such that $R_k\longrightarrow+\infty$ as $k\longrightarrow+\infty$ and denote $\psi_k:=\psi_{_{R_k}}$ and $V_{k}:=V_{R_k}$. Consequently, using the fact that $u$ is a solution,  the definition of $V_R$, $\rho(0)=1$ and $x\cdot \nu=0$ in $\Gamma_{R_k}$, we have
		\begin{align}\label{limite}
			\int_{\Gamma_{R_k}}g(u)(x\cdot\nabla u) \psi_k\,\mathrm{d}x^\prime=&\int_{\Gamma_{R_k}}V_{k} \cdot \nu \mathrm{d}x^\prime +\int_{\Gamma_{R_k}} \frac{1}{2}|\nabla u|^2(x\cdot \nu)\psi_k \mathrm{d}x^\prime \nonumber\\
            =& \int_{\Gamma_{R_k}}V_{k} \cdot \nu \mathrm{d}x^\prime\nonumber\\
            =&-\int_{ B_{R_k}^+}\left[f(u)(x\cdot\nabla u)+\frac{N-2}{2}\rho(x_N)|\nabla u|^2+\frac{\rho'(x_N)x_N}{2}|\nabla u|^2\right]\psi_{R}\,\mathrm{d}x\nonumber\\
            +& o_{R_k}(1).
		\end{align}
		We now claim that
        \begin{equation}\label{f limit}
            \int_{B_{R_k}^+}f(u)(x\cdot \nabla u)\psi_R\mathrm{d}x\longrightarrow -N\int_{\mathds{R}^N_+}F(u)\mathrm{d}x,\quad \text{as}\quad k\longrightarrow\infty.
        \end{equation}
        Indeed, integrating by parts, we obtain
		\begin{align*}
			\int_{B_{R_k}^+}f(u)(x\cdot\nabla u)\psi_k\,\mathrm{d}x	&=\sum_{i=1}^N\int_{B_{R_k}^+}\left(F(u)\right)_{x_i}x_i\psi_k\,\mathrm{d}x\\
			&=\int_{B_{R_k}^+}\left[-NF(u)\psi_k-F(u)(x\cdot\nabla\psi_k)\right]\,\mathrm{d}x.
		\end{align*}
		By the Dominated Convergence Theorem,
		$$
\int_{B_{R_k}^+}F(u)\psi_k\,\mathrm{d}x\longrightarrow\int_{\mathds{R}^N_+}F(u)\,\mathrm{d}x, \quad \text{as}\quad k\longrightarrow\infty.
		$$
Furthermore
		\begin{align*}
		\left|\int_{B_{R_k}^+}F(u)(x\cdot\nabla\psi_k)\,\mathrm{d}x\right|&\leq \int_{B_{R_k}^+\backslash B_{R_k/2}}|F(u)||x||\nabla\psi_k|\,\mathrm{d}x\\&\leq 2\int_{B_{R_k}^+\backslash B_{R_k/2}}|F(u)|\,\mathrm{d}x\longrightarrow0,\quad \text{as}\quad k\longrightarrow\infty,   
		\end{align*}
since $F(u)\in L^1(\mathds{R}^N_+)$, which concludes \eqref{f limit}. Now, observe that 
		$$
		\begin{aligned}
		\int_{\Gamma_{R_k}}g(u)(x\cdot\nabla u) \psi_k\,\mathrm{d}x^\prime&=\sum_{i=1}^{N-1}\int_{\Gamma_{R_k}}\left(G(u)\right)_{x_i}x_i\psi_k\,\mathrm{d}x^\prime\\
		&=\int_{\Gamma_{R_k}}\left[-(N-1)G(u)\psi_k-G(u)\left(x^\prime\cdot
		\nabla_{x^\prime}\psi_k\right)\right]\,\mathrm{d}x^\prime.
		\end{aligned}
		$$
		By the same reasoning, we obtain
		$$
\int_{\Gamma_{R_k}}G(u)\psi_k\,\mathrm{d}x^\prime\longrightarrow \int_{\mathds{R}^{N-1}}G(u)\,\mathrm{d}x^\prime\quad\mbox{and}\quad \int_{\Gamma_{R_k}}G(u)\left(x^\prime\cdot
		\nabla_{x^\prime}\psi_k\right)\,\mathrm{d}x^\prime\longrightarrow 0.
		$$
Hence,
\begin{equation}\label{g limit}
    \int_{\Gamma_{R_k}}g(u)(x\cdot\nabla u) \psi_k\,\mathrm{d}x^\prime \longrightarrow -(N-1)\int_{\mathds{R}^{N-1}}G(u)\mathrm{d}x^\prime.
\end{equation}
		Furthermore, by the dominated convergence Theorem,
		\begin{equation}\label{rho limit}
		\int_{B_{R_k}^+}\rho(x_N)|\nabla u|^2\psi_k\,\mathrm{d}x\longrightarrow\int_{\mathds{R}^N_+}\rho(x_N)|\nabla u|^2\,\mathrm{d}x.
		\end{equation}
        Finally, by $(\rho_1)$ and the dominated convergence theorem, we get
        
        \begin{equation}\label{rho' limit}
            \int_{B_{R_k}^+}\rho^\prime(x_N)x_N|\nabla u|^2\psi_k\,\mathrm{d}x\longrightarrow\int_{\mathds{R}^N_+}\rho^\prime(x_N)x_N|\nabla u|^2\,\mathrm{d}x.
        \end{equation}
Then, passing to the limit in \eqref{limite} and using \eqref{f limit}, \eqref{g limit}, \eqref{rho limit} and \eqref{rho' limit}, we obtain
\begin{equation*}
    -(N-1)\int_{\mathds{R}^{N-1}}G(u)\mathrm{d}x^\prime=N\int_{\mathds{R}^N_+}F(u)\mathrm{d}x-\frac{N-2}{2}\int_{\mathds{R}^N_+}\rho(x_N)|\nabla u|^2\mathrm{d}x-\frac{1}{2}\int_{\mathds{R}^N_+}\rho'(x_N)x_N |\nabla u|^2\mathrm{d}x.
\end{equation*}
Therefore,
\begin{equation*}
    \frac{N-2}{2}\int_{\mathds{R}_{+}^{N}}\rho(x_N)|\nabla u|^2\mathrm{d}x+\frac{1}{2}\int_{\mathds{R}_{+}^{N}}\rho'(x_N)x_N|\nabla u|^2\mathrm{d}x=N\int_{\mathds{R}^N_+}F(u)\mathrm{d}x+(N-1)\int_{\mathds{R}^{N-1}}G(u)\mathrm{d}x^\prime.
\end{equation*}

\end{proof}

Now, we are able to prove the nonexistence results.

\begin{proof}[Proof of Theorem \ref{Nonexistence1}]: For each item of the theorem, given $u$ be a nonnegative weak solution for \eqref{PG}, we have $u^q\in L^1(\mathds{R}^{N-1})$ and $u^p\in L^1(\mathds{R}^N_+)$. Moreover, by Theorem \ref{H2_loc regularity theorem}, we have $u\in H^2_{\mathrm{loc}}(\overline{\mathds{R}^N_+})$, so the Pohozaev identity in Theorem \ref{Pohozaev} applies and yields:
\begin{align*}
     \frac{N-2}{2}\int_{\mathds{R}_{+}^{N}}\rho(x_N)|\nabla u|^2\mathrm{d}x+\frac{1}{2}\int_{\mathds{R}_{+}^{N}}\rho'(x_N)x_N|\nabla u|^2\mathrm{d}x=&\frac{aN}{p}\int_{\mathds{R}^N_+}|u|^p\mathrm{d}x\\
     +&\frac{b(N-1)}{q}\int_{\mathds{R}^{N-1}}|u|^q\mathrm{d}x^\prime.
\end{align*}
Furthermore, taking $u$ as a testing function, we obtain
\begin{equation*}
    \int_{\mathds{R}^N_+}\rho(x_N)|\nabla u|^2\mathrm{d}x=a\int_{\mathds{R}^N_+}|u|^p\mathrm{d}x+b\int_{\mathds{R}^{N-1}}|u|^q\mathrm{d}x^\prime.
\end{equation*}
By combining these two equations, we get the following relation:
\begin{align*}
    \frac{1}{2}\int_{\mathds{R}^N_+}\rho'(x_N)x_N|\nabla u|^2\mathrm{d}x&=a\left(\frac{N}{p}-\frac{N-2}{2}\right)\int_{\mathds{R}^{N}_+}|u|^p\mathrm{d}x\\
    &+b\left(\frac{N-1}{q}-\frac{N-2}{2}\right)\int_{\mathds{R}^{N-1}}|u|^q\mathrm{d}x^\prime.
\end{align*}
Since $\rho^\prime(x_N)x_N>0$ by $(\rho_1)$, the left-hand side of the equation is strictly positive unless $u\equiv 0$. Under conditions of the itens $(i)$ or $(ii)$,  the right-hand side becomes nonpositive. Therefore, the only possibility is $u\equiv 0$, which concludes the proof.
\end{proof}

\begin{proof}[Proof of Theorem \ref{Nonexistence2}]: Lets prove the item $(iii)$ and the same analysis can be applied to the other cases. If $u \in \mathcal{D}^{1,2}_\rho(\mathds{R}^N_+)$ and $(\rho_0)$ holds with $\gamma>0$, we have $u^{2^\ast}\in L^1(\mathds{R}^N_+)$ by the classical Sobolev inequality, and $u^{2_\ast}\in L^1(\mathds{R}^{N-1})$ by \eqref{ineq}. Then, since $u\in H^2_{\mathrm{loc}}(\overline{\mathds{R}^N_+})$, we can argue as in the previous theorem and use the fact that $p=2^\ast$ and $q=2_\ast$ to obtain
\begin{align*}
    \frac{1}{2}\int_{\mathds{R}^N_+}\rho'(x_N)x_N|\nabla u|^2\mathrm{d}x&=0.
\end{align*}
Given that $\rho^\prime(x_N)x_N>0$ by $(\rho_1)$, it remains that $u\equiv 0$, which concludes the proof.    
\end{proof}

\bigskip

%%%%%%%%%%%%%%%%%%%%%%%%%%%%%%%%%%%%%%%%%%%%%%%%%%%%%%%%%%%%%%%%%%%%%%%%%%%%%%%%%%%%%%%%%%%%%%%%%%
	%  Statements and declarations
%%%%%%%%%%%%%%%%%%%%%%%%%%%%%%%%%%%%%%%%%%%%%%%%%%%%%%%%%%%%%%%%%%%%%%%%%%%%%%%%%%%%%%%%%%%%%%%%%%

\subsection*{Statements and declarations}

\begin{flushleft}
	{\bf Funding:}  This work was supported partially by CAPES MATH AMSUD grant 88887.878894/2023-00.
	J. M. do \'O acknowledges partial support from CNPq through grants 312340/2021-4, 409764/2023-0, 443594/2023-6,  and E. Medeiros acknowledges partial support from CNPq through grant 310885/2023-0
	and Para\'iba State Research Foundation (FAPESQ), grant no 3034/2021.  \\
	{\bf Ethical Approval:}  Not applicable.\\
	{\bf Competing interests:}  Not applicable. \\
	{\bf Authors' contributions:}    All authors contributed to the study conception and design. All authors performed material preparation, data collection, and analysis. The authors read and approved the final manuscript.\\
	{\bf Availability of data and material:}  Not applicable.\\
	{\bf Ethical Approval:}  All data generated or analyzed during this study are included in this article.\\
	{\bf Consent to participate:}  All authors consent to participate in this work.\\
	{\bf Conflict of interest:} The authors declare no conflict of interest. \\
	{\bf Consent for publication:}  All authors consent for publication. \\
\end{flushleft}

\bigskip
	
%%%%%%%%%%%%%%%%%%%%%%%%%%%%%%%%%%%%%%%%%%%%%%%%%%%%%%%%%%%%%%%%%%%%%%%%%%%%%%%%%%%%%%%%%%%%%%%%%%
	%  REFERENCES
%%%%%%%%%%%%%%%%%%%%%%%%%%%%%%%%%%%%%%%%%%%%%%%%%%%%%%%%%%%%%%%%%%%%%%%%%%%%%%%%%%%%%%%%%%%%%%%%%%


\begin{thebibliography}{99999}

\bibitem{CAOM} E. Abreu, R. Clemente, J. M. do Ó, E. Medeiros, \textit{p-harmonic functions in the upper half-space}, Potential Anal. \textbf{60} (2024), no. 4, 1383–1406.

\bibitem{AFM} E. Abreu, D. Felix, E. Medeiros, \textit{A weighted Hardy type inequality and its applications}, Bull. Sci. Math. \textbf{166} (2021), Paper No. 102937, 25 pp.

\bibitem{Abreu-Furtado-Medeiros} E. Abreu, M. Furtado, E. Medeiros, \textit{Remarks on a Sobolev embedding}, Appl. Math. Lett. \textbf{147} (2024), Paper No. 108846, 6 pp.

\bibitem{Abreu-Furtado-Medeiros 2} E. Abreu, M. Furtado, E. Medeiros, \textit{On a Hardy–Sobolev inequality with remainder term and its consequences}, Preprint.

\bibitem{Abreu-JM-Medeiros} E. Abreu, J. M. do Ó, E. Medeiros, \textit{Properties of positive harmonic functions on the half-space with a nonlinear boundary condition}, J. Differential Equations \textbf{248} (2010), 617–637.

\bibitem{ADN 1} S. Agmon, A. Douglis, L. Nirenberg, \textit{Estimates near the boundary for solutions of elliptic partial differential equations satisfying general boundary conditions. I}, Comm. Pure Appl. Math. \textbf{12} (1959), 623–727.

\bibitem{Claudianor-Angelo} C. O. Alves, A. R. F. de Holanda, \textit{A Berestycki–Lions type result for a class of degenerate elliptic problems involving the Grushin operator}, arXiv:2109.01633 [math.AP], 2021.

\bibitem{Aronson} D. G. Aronson, H. F. Weinberger, \textit{Multidimensional nonlinear diffusion arising in population genetics}, Adv. Math. \textbf{30} (1978), 33–76.

\bibitem{Badiale-Tarantello} M. Badiale, G. Tarantello, \textit{A Sobolev–Hardy inequality with applications to a nonlinear elliptic equation arising in astrophysics}, Arch. Ration. Mech. Anal. \textbf{163} (2002), no. 4, 259–293.

\bibitem{Beckner} W. Beckner, \textit{Sharp Sobolev inequalities on the sphere and the Moser–Trudinger inequality}, Ann. of Math. (2) \textbf{138} (1993), 213–242.

\bibitem{Barrios} B. Barrios, E. Colorado, R. Servadei, F. Soria, \textit{A critical fractional equation with concave–convex power nonlinearities}, Ann. Inst. H. Poincaré Anal. Non Linéaire \textbf{32} (2015), no. 4, 875–900.

\bibitem{Berestycki-Lions} H. Berestycki, P. L. Lions, \textit{Nonlinear scalar field equations. I. Existence of a ground state}, Arch. Ration. Mech. Anal. \textbf{82} (1983), no. 4, 313–345.

\bibitem{Brezis} H. Brezis, \textit{Functional Analysis, Sobolev Spaces and Partial Differential Equations}, Springer, New York, 2010.

\bibitem{Chandra} S. Chandrasekhar, \textit{An Introduction to the Theory of Stellar Structure}, Univ. of Chicago Press, 1939; reprinted by Dover, New York, 1957.

\bibitem{Chipot-Fila-Shafir1} M. Chipot, M. Chlebík, M. Fila, I. Shafrir, \textit{Existence of positive solutions of a semilinear elliptic equation in \(\mathds{R}^n_+\) with a nonlinear boundary condition}, J. Math. Anal. Appl. \textbf{223} (1998), no. 2, 429–471.

\bibitem{Chipot-Fila-Shafir2} M. Chipot, I. Shafrir, M. Fila, \textit{On the solutions to some elliptic equations with nonlinear Neumann boundary conditions}, Adv. Differential Equations \textbf{1} (1996), no. 1, 91–110.

\bibitem{Do-Freire-Medeiros} J. M. do Ó, R. Freire, E. Medeiros, \textit{Liouville-type theorems and existence of solutions for quasilinear elliptic problems}, Preprint.

\bibitem{Do-Freire-Medeiros 2} J. M. do Ó, R. Freire, E. Medeiros, \textit{Weighted Sobolev trace embeddings and their applications}, Preprint.

\bibitem{Drabek} P. Drábek, A. Kufner, F. Nicolosi, \textit{Quasilinear Elliptic Equations with Degenerations and Singularities}, de Gruyter, Berlin, 1997.

\bibitem{Escobar1} J. F. Escobar, \textit{Sharp constant in a Sobolev trace inequality}, Indiana Univ. Math. J. \textbf{37} (1988), 687–698.

\bibitem{Escobar2} J. F. Escobar, \textit{Uniqueness theorems on conformal deformation of metrics, Sobolev inequalities, and an eigenvalue estimate}, Comm. Pure Appl. Math. \textbf{43} (1990), no. 7, 857–883.

\bibitem{Esteban-Lions} M. J. Esteban, P. L. Lions, \textit{Existence and nonexistence results for semilinear elliptic problems in unbounded domains}, Proc. Roy. Soc. Edinburgh Sect. A \textbf{93} (1982/83), no. 1–2, 1–14.

\bibitem{GT} D. Gilbarg, N. S. Trudinger, \textit{Elliptic Partial Differential Equations of Second Order}, 2nd ed., Springer-Verlag, Berlin, 1983.

\bibitem{Guedda-Veron} M. Guedda, L. Véron, \textit{Quasilinear elliptic equations involving critical Sobolev exponents}, Nonlinear Anal. \textbf{13} (1989), no. 8, 879–902.

\bibitem{Ilyasov-Takac} Y. Il’yasov, P. Takáč, \textit{Optimal \(W^{2,2}_{\mathrm{loc}}\)-regularity, Pohožaev’s identity, and nonexistence of weak solutions to some quasilinear elliptic equations}, J. Differential Equations \textbf{252} (2012), 2792–2822.

\bibitem{Kawohl} B. Kawohl, \textit{Rearrangements and Convexity of Level Sets in PDE}, Lecture Notes in Math., vol. 1150, Springer-Verlag, Berlin, 1985.

\bibitem{Li-Ni} Y. Li, W. Ni, \textit{On the existence and symmetry properties of finite total mass solutions of the Matukuma equation, the Eddington equation and their generalizations}, Arch. Ration. Mech. Anal. \textbf{108} (1989), no. 2, 175–194.

\bibitem{Li-Zhu} Y. Li, M. Zhu, \textit{Uniqueness theorems through the method of moving spheres}, Duke Math. J. \textbf{80} (1995), no. 2, 383–417.

\bibitem{YanYan-Lei} Y. Y. Li, L. Zhang, \textit{Liouville-type theorems and Harnack-type inequalities for semilinear elliptic equations}, J. Anal. Math. \textbf{90} (2003), 27–87.

\bibitem{Lieberman} G. M. Lieberman, \textit{Boundary regularity for solutions of degenerate elliptic equations}, Nonlinear Anal. \textbf{12} (1988), no. 11, 1203–1219.

\bibitem{PLLions} P. L. Lions, \textit{Symétrie et compacité dans les espaces de Sobolev}, J. Funct. Anal. \textbf{49} (1982), 315–334.

\bibitem{Lions} P. L. Lions, \textit{The concentration–compactness principle in the calculus of variations. The limit case. I}, Rev. Mat. Iberoamericana \textbf{1} (1985), 145–201.

\bibitem{Liu-Liu} J. Liu, X. Liu, \textit{On the eigenvalue problem for the p-Laplacian operator in \(\mathds{R}^N\)}, J. Math. Anal. Appl. \textbf{379} (2011), no. 2, 861–869.

\bibitem{mawhin-willem} J. Mawhin, M. Willem, \textit{Critical Point Theory and Hamiltonian Systems}, Appl. Math. Sci., vol. 74, Springer-Verlag, New York, 1989.

\bibitem{Lina} L. Meinecke, \textit{Multiscale modeling of diffusion in a crowded environment}, Bull. Math. Biol. \textbf{79} (2017), no. 11, 2672–2695.

\bibitem{Pohozaev} S. I. Pohožaev, \textit{On the eigenfunctions of the equation \(\Delta u + \lambda f(u) = 0\)}, Dokl. Akad. Nauk SSSR \textbf{165} (1965), 36–39.

\bibitem{Pucci-Serrin} P. Pucci, J. Serrin, \textit{A general variational identity}, Indiana Univ. Math. J. \textbf{35} (1986), no. 3, 681–703.

\bibitem{Rudin} W. Rudin, \textit{Real and Complex Analysis}, 3rd ed., McGraw-Hill, New York, 1987.

\bibitem{Trudinger} N. S. Trudinger, \textit{On Harnack type inequalities and their application to quasilinear elliptic equations}, Comm. Pure Appl. Math. \textbf{20} (1967), 721–747.

\end{thebibliography}
\end{document}